\documentclass[11pt]{amsart}
\usepackage{amssymb,graphicx}

\title[Macdonald Polynomials and the Charge]{From Macdonald Polynomials to a Charge Statistic beyond Type $A$}

\author{Cristian Lenart}

\address{Department of Mathematics and Statistics, State University of New York at Albany, Albany, NY 12222}
\email{lenart@albany.edu}



\thanks{Cristian Lenart was partially supported by the National Science Foundation grant  DMS-0701044}

\DeclareMathOperator{\rev}{rev}

\newlength{\cellsize}
\newcommand\tableau[1]{
\vcenter{
\let\\=\cr
\baselineskip=-16000pt
\lineskiplimit=16000pt
\lineskip=0pt
\halign{&\tableaucell{##}\cr#1\crcr}}}

\newcommand{\tableaucell}[1]{{%
\def \arg{#1}\def \void{}%
\ifx \void \arg
\vbox to \cellsize{\vfil \hrule width \cellsize height 0pt}%
\else
\unitlength=\cellsize
\begin{picture}(1,1)
\put(0,0){\makebox(1,1){$#1$}}
\put(0,0){\line(1,0){1}}
\put(0,1){\line(1,0){1}}
\put(0,0){\line(0,1){1}}
\put(1,0){\line(0,1){1}}
\end{picture}%
\fi}}

\numberwithin{equation}{section}

\theoremstyle{plain}
\newtheorem{theorem}{Theorem}[section]
\newtheorem{proposition}[theorem]{Proposition}
\newtheorem{lemma}[theorem]{Lemma}
\newtheorem{corollary}[theorem]{Corollary}

\newtheorem{definition}[theorem]{Definition}

\newtheorem{example}[theorem]{Example}

\newtheorem{algorithm}[theorem]{Algorithm}

\theoremstyle{remark}
\newtheorem{remark}[theorem]{Remark}
\newtheorem{remarks}[theorem]{Remarks}
\newtheorem{condition}[theorem]{Condition}

\def\R{\mathbb{R}}
\def\Z{\mathbb{Z}}

\def\charge{{\rm charge}}
\def\cw{{\rm cw}}
\def\rev{{\rm rev}}
\def\sign{{\rm sign}}
\def\content{{\rm content}}
\def\level{{\rm level}}
\def\weight{{\rm weight}}
\def\endch{{\rm end}}
\def\Waff{W_{\mathrm{aff}}}

\def\h{\mathfrak{h}}
\def\hR{\mathfrak{h}^*_\mathbb{R}}

\newcommand{\casetwo}[3]{\left\{ \begin{array}{ll} #1 &\mbox{if $#2$} \\[2mm] #3 &\mbox{otherwise}\,. \end{array} \right.}

\begin{document}

\bibliographystyle{plain}

\begin{abstract} The charge is an intricate statistic on words, due to Lascoux and Sch\"utzenberger, which gives positive combinatorial formulas for Lusztig's $q$-analogue of weight multiplicities and the energy function on affine crystals, both of type $A$. As these concepts are defined for all Lie types, it has been a long-standing problem to express them based on a generalization of charge. I present a method for addressing this problem in classical Lie types, based on the recent Ram-Yip formula for Macdonald polynomials and the quantum Bruhat order on the corresponding Weyl group. The details of the method are carried out in type $A$ (where we recover the classical charge) and type $C$ (where we define a new statistic). 
\end{abstract}


\maketitle

\section{Introduction}


Charge is a statistic on words with partition content, which was originally defined (in type $A$) by Lascoux and Sch\"utzenberger~\cite{lassuc}. It is calculated by enumerating certain cycles in the given word, see 
Section~\ref{redcha}. The original application of charge was to give a positive combinatorial formula the Kostka--Foulkes polynomials, or Lusztig's $q$-analogue of weight multiplicities \cite{lusscq}.  
Nakayashiki and Yamada~\cite{naykpe} showed that the charge also expresses the energy function on affine crystals of type $A$ (the energy function is an important grading used in one-dimensional configuration sums~\cite{HKOTT:2002, HKOTY:1999}, as well as in other areas). 
Defining a charge statistic beyond type $A$ and extending some of the above applications has been a long-standing problem. 
Lecouvey~\cite{leckfp,lecccg} defined such a statistic on Kashiwara-Nakashima tableaux of types $B$, $C$, and $D$ \cite{kancgr}, based on the intricate combinatorics of the corresponding plactic monoids; but he was only able to relate his charge to the corresponding Kostka--Foulkes polynomials in very special cases.

In this paper we propose a method for defining a charge statistic in classical types. The details are carried out in types $A$ and $C$; in type $A$ we recover the classical charge, and in type $C$ we obtain a new statistic. This method is based on the theory of Macdonald polynomials and, in particular, on the Ram-Yip formula for these polynomials. The (symmetric) Macdonald polynomials \cite{macaha} are a remarkable family of orthogonal symmetric polynomials depending on parameters $q,t$. They are associated to any finite root system, and they generalize the corresponding irreducible characters, which are recovered upon setting $q=t=0$. Upon setting $q=0$, the Macdonald polynomials specialize to the corresponding Hall-Littlewood polynomials (or spherical functions for a $p$-adic group) \cite{macsfg}. Upon setting $t=0$ in simply-laced types, we obtain certain affine Demazure characters \cite{ionnmp}. The Ram-Yip formula \cite{raycfm} is a monomial formula for Macdonald polynomials of arbitrary type, which is expressed in terms of combinatorial objects called alcove walks; the latter originate in the work of Gaussent-Littelmann \cite{gallsg}  and Lenart-Postnikov \cite{lapawg,lapcmc} on combinatorial models in the representation theory of Lie algebras, namely the LS-gallery and the alcove model, respectively. 

We now present the outline of our method. We start by specializing the Ram-Yip formula to $t=0$, and by observing that the alcove walks that survive correspond to paths in the quantum Bruhat graph. This graph first arose in connection with the quantum cohomology of flag varieties \cite{fawqps}, and is obtained by adding extra down edges to the Hasse diagram of the Bruhat order on the Weyl group. Our main result is the construction of a bijection, in types $A$ and $C$, between the mentioned paths and certain column-strict fillings of Young diagrams. More precisely, these fillings are tensor products of the corresponding Kashiwara-Nakashima columns \cite{kancgr}; the latter index the basis of a fundamental representation in classical types, or the vertices of the corresponding crystals (Kashiwara's crystals encode the structure of quantum group representations as the quantum parameter goes to $0$). Based on the mentioned bijection, we translate the statistic in the Ram-Yip formula to a statistic on the mentioned fillings, which is our charge. This method highlights the fact that the charge construction (including the classical one) is based in a subtle way on the quantum Bruhat graph.

In addition to the charge construction, this paper achieves two more goals: (i) we give charge formulas for Macdonald polynomials of types $A$ and $C$ at $t=0$; (ii) we realize tensor products of type $A$ and $C$ crystals in terms of an extension of the alcove model \cite{lapcmc} based on paths in the quantum Bruhat graph (the original model is for highest weight crystals, and is based on saturated chains in Bruhat order). 

The immediate application of the type $C$ charge in this paper is that it expresses the corresponding energy function, thus generalizing the type $A$ result of Nakayashiki-Yamada mentioned above. This is proved in \cite{lascef}, based on the recent reinterpretation 
in~\cite{ST:2011} of the energy function as a grading of certain affine Demazure crystals. Thus, we have a simple combinatorial method to compute the energy function in type $C$, in addition to type $A$. Note that, from a computational perspective, neither the original definition of the energy (based on the local energy and the  combinatorial $R$-matrix), nor the construction in~\cite{ST:2011} mentioned above are very efficient. Generalizations of the results in this paper as well as of \cite{lascef} to types $B$ and $D$ will be pursued in future publications.

Another application of the results in this paper is the work in progress with A. Lubovsky on defining the main constructions of the alcove model (crystal operators, certain combinatorial transformations called Yang-Baxter moves \cite{lenccg}) for its extension mentioned above, related to tensor products of crystals. In this way, for instance, the Yang-Baxter moves would provide an efficient construction of the combinatorial $R$-matrix (for commuting tensor factors).

We now present the outline of the paper. In Section \ref{alcmod} we recall the alcove model and specialize the Ram-Yip formula to $t=0$ in  full generality. In Sections \ref{secta} and \ref{sectc} we specialize the alcove model and the above formula to types $A$ and $C$, respectively. In Sections \ref{chrev} and \ref{chc} we derive the classical (type $A$) charge and the new type $C$ charge via the method outlined above. Most of the effort goes towards constructing the inverse of the bijection mentioned above, between paths in the quantum Bruhat graph and fillings of Young diagrams. Two crucial results needed in this construction are proved in Sections \ref{coll} and \ref{colr}, respectively. In Section \ref{typebd} we discuss the additional complexity of our method in types $B$ and $D$, and describe the conjectured construction of the corresponding charge. 

\section{The alcove model and the Ram-Yip formula}\label{alcmod}

We start with some background information on finite root systems and the alcove model. Then we state the Ram-Yip formula at $t=0$.

\subsection{Root systems}\label{rootsyst}

Let $\mathfrak{g}$ be a complex semisimple Lie algebra, and $\h$ a Cartan subalgebra, whose rank is $r$.
Let $\Phi\subset \h^*$ be the 
corresponding irreducible {\it root system}, $\hR\subset \h^*$ the real span of the roots, and $\Phi^+\subset \Phi$ the set of positive roots. Let $\rho:=\frac{1}{2}(\sum_{\alpha\in\Phi^+}\alpha)$. 
Let $\alpha_1,\ldots,\alpha_r\in\Phi^+$ be the corresponding 
{\it simple roots}.
We denote by $\langle\,\cdot\,,\,\cdot\,\rangle$ the nondegenerate scalar product on $\hR$ induced by
the Killing form.  
Given a root $\alpha$, we consider the corresponding {\it coroot\/} $\alpha^\vee := 2\alpha/\langle\alpha,\alpha\rangle$ and reflection $s_\alpha$.  

Let $W$ be the corresponding  {\it Weyl group\/}, whose Coxeter generators are denoted, as usual, by $s_i:=s_{\alpha_i}$. The length function on $W$ is denoted by $\ell(\,\cdot\,)$. The {\em Bruhat order} on $W$ is defined by its covers $w\lessdot ws_\alpha$, for $\alpha\in\Phi^+$, if $\ell(ws_\alpha)=\ell(w)+1$. The mentioned covers correspond to the labeled directed edges of the {\em Bruhat graph} on $W$: 
\begin{equation}\label{brgr}w\stackrel{\alpha}{\longrightarrow} ws_\alpha\;\;\;\;\mbox{for}\;\;
w\lessdot ws_\alpha\,. \end{equation}

The {\it weight lattice\/} $\Lambda$ is given by
\begin{equation}
\Lambda:=\{\lambda\in \hR \::\: \langle\lambda,\alpha^\vee\rangle\in\Z
\textrm{ for any } \alpha\in\Phi\}.
\label{eq:weight-lattice}
\end{equation}
The weight lattice $\Lambda$ is generated by the 
{\it fundamental weights\/}
$\omega_1,\ldots,\omega_r$, which form the dual basis to the 
basis of simple coroots, i.e., $\langle\omega_i,\alpha_j^\vee\rangle=\delta_{ij}$.
The set $\Lambda^+$ of {\it dominant weights\/} is given by
$$
\Lambda^+:=\{\lambda\in\Lambda \::\: \langle\lambda,\alpha^\vee\rangle\geq 0
\textrm{ for any } \alpha\in\Phi^+\}.
$$
Let $\Z[\Lambda]$ be the group algebra of the weight lattice $\Lambda$, which  has
a $\Z$-basis of formal exponents $\{x^\lambda \::\: \lambda\in\Lambda\}$ with
multiplication $x^\lambda\cdot x^\mu := x^{\lambda+\mu}$.

Given  $\alpha\in\Phi$ and $k\in\Z$, we denote by $s_{\alpha,k}$ the reflection in the affine hyperplane
\begin{equation}
H_{\alpha,k} := \{\lambda\in \hR \::\: \langle\lambda,\alpha^\vee\rangle=k\}.
\label{eqhyp}
\end{equation}
These reflections generate the {\it affine Weyl group\/} $\Waff$ for the {\em dual root system} 
$\Phi^\vee:=\{\alpha^\vee \::\: \alpha\in\Phi\}$. 
The hyperplanes $H_{\alpha,k}$ divide the real vector space $\hR$ into open
regions, called {\it alcoves.} 
The {\it fundamental alcove\/} $A^\circ$ is given by 
$$
A^\circ :=\{\lambda\in \hR \::\: 0<\langle\lambda,\alpha^\vee\rangle<1 \textrm{ for all }
\alpha\in\Phi^+\}.
$$

\subsection{The alcove model}\label{alcovewalks}

We say that two alcoves are {\it adjacent} 
if they are distinct and have a common wall.  
Given a pair of adjacent alcoves $A$ and $B$, we write 
$A\stackrel{\beta}\longrightarrow B$ if the common wall 
is of the form $H_{\beta,k}$ and the root $\beta\in\Phi$ points 
in the direction from $A$ to $B$.  

\begin{definition}
A (reduced) {\em alcove path\/} is a sequence of alcoves 
\[A_0\stackrel{\beta_1}\longrightarrow A_1\stackrel{\beta_2}\longrightarrow \ldots
\stackrel{\beta_m}\longrightarrow A_{m}=A_*\]
of minimal length among such all sequences which connect  $A_0$ and $A_*$. If $A_0=A^\circ$ and $A_*=A^\circ+\mu$, we call the sequence of roots $(\beta_1,\ldots,\beta_m)$ a {\em $\mu$-chain} (of roots). 
\end{definition}

We now fix a dominant weight $\mu$ and an alcove path $\Pi=(A_0,\ldots,A_m)$ from $A_0=A^\circ$ to $A_m=A^\circ+\mu$. Note that $\Pi$ is determined by the corresponding $\mu$-chain of {\em positive} roots $\Gamma:=(\beta_1,\ldots,\beta_m)$. We let $r_i:=s_{\beta_i}$, and let $\widehat{r}_i$ be the affine reflection in the affine hyperplane containing the common face of $A_{i-1}$ and $A_i$, for $i=1,\ldots,m$; in other words, $\widehat{r}_i:=s_{\beta_i,l_i}$, where $l_i:=\#\{j\le i\::\: \beta_j = \beta_i\}$ is the cardinality of the corresponding set. We will call $l_i$ the {\em affine level} of the $i$th root in $\Gamma$. 

\begin{remark} The original definition of $\mu$-chains in \cite{lapawg,lapcmc} corresponds to the reverse of the $\mu$-chains in this paper. Two equivalent definitions of the original $\mu$-chains (in terms of reduced words in affine Weyl groups, and an interlacing condition) can be found in \cite{lapawg}[Definition 5.4] and \cite{lapcmc}[Definition 4.1 and Proposition 4.4]. 
\end{remark}

\begin{definition} A {\em folding pair} is a pair $(w,J)$, where $w$ is a Weyl group element and $J=\{j_1<\ldots<j_s\}\subseteq[m]:=\{1,\ldots,m\}$. The set $W\times 2^{[m]}$ of folding pairs is denoted by ${\mathcal F}(\Gamma)$.
\end{definition}

A folding pair $(w,J)=(w,\{j_1<\ldots<j_s\})$ is identified with the sequence of Weyl group elements
\begin{equation}\label{alcchain}\pi(w,J):=(w_0,w_1,\ldots,w_s)\,,\;\;\mbox{where}\;\;w_i:=wr_{j_1}\ldots r_{j_i}\,.\end{equation}
In particular, $w_0=w$, and we let $\endch(w,J):=w_s$. We also associate with $(w,J)$ the following weight:
\begin{equation}\label{defphimu}\weight(w,J):=w\widehat{r}_{j_1}\ldots \widehat{r}_{j_s}(\mu)\,.\end{equation}

\begin{definition} Given a folding pair $(w,J)=(w,\{j_1<\ldots<j_s\})$, the element $j_i$ is called a {\em positive folding} (resp. {\em negative folding}) if $w_{i-1}>w_i$ (resp. $w_{i-1}<w_i$), cf. {\rm (\ref{alcchain})}. The set of positive (resp. negative) foldings is denoted by $J^+$ (resp. $J^-$). 
\end{definition}

The above terminology can be explained as follows. Start with the alcove path $w(\Pi):=(w(A_0),w(A_1),\ldots,w(A_m))$ and successively reflect the tail of the current sequence of alcoves in the affine hyperplane containing the common face of the alcoves $w(A_{j_i-1})$ and $w(A_{j_{i}})$, for $i=s,s-1,\ldots,1$. The result is a sequence of alcoves $(A_0',\ldots,A_m')$, which is called an {\em alcove walk}. Note that $A_{j_i-1}'=A_{j_i}'$. We have $j_i\in J^+$ (resp. $j_i\in J^-$) if the alcove $A_{j_i-1}'=A_{j_i}'$ is on the positive (resp. negative) side of the corresponding reflecting hyperplane. For more details, we refer to \cite{lencfm}[Section 2.2].

\begin{example}\label{exalcwalk} {\rm Consider the dominant weight $\mu=3\varepsilon_1+\varepsilon_2$ in the root system $A_2$ (cf. Section \ref{seta} and the notation therein, which is used freely in this example). In the figure below, the alcoves $A^\circ$ and $A^\circ+\mu$ are shaded, and a reduced alcove path connecting them is shown. The corresponding $\mu$-chain is $(\alpha_{13},\alpha_{12},\alpha_{13},\alpha_{23},\alpha_{13},\alpha_{12})$. 

For $(w,J)=(123,\{1,2\})$, the associated sequence of permutations, written in one-line notation, as a chain in Bruhat order, is $(123<321>231)$, since $r_1=(1,3)$ and $r_2=(1,2)$. Thus $J^+=\{2\}$ and $J^-=\{1\}$, which can also be checked by folding the alcove path in the figure above along the affine hyperplanes corresponding to $\widehat{r}_2$ and $\widehat{r}_1$, in this order. We also have $\weight(w,J)=\varepsilon_2$, which can be easily checked via the mentioned folding, as well.

\vspace{-0.4cm}

\[\mbox{\includegraphics[scale=0.52]{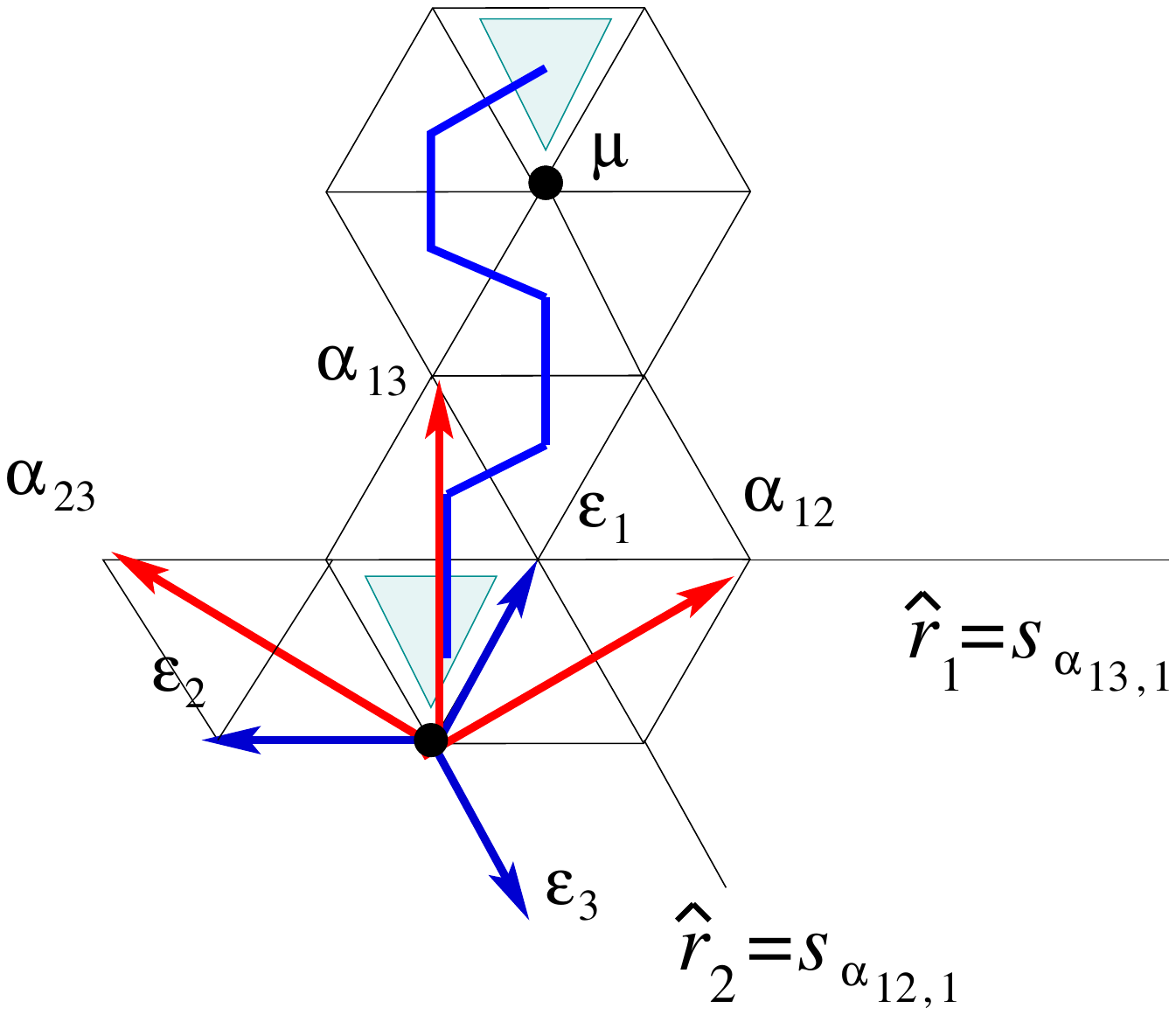}}\]

\vspace{-7.6cm}
}
\end{example}

\subsection{The Ram-Yip formula at $t=0$ and the quantum Bruhat graph} Recall that we fixed a dominant weight $\mu$ and a $\mu$-chain $\Gamma$. We now specialize the Ram-Yip formula \cite{raycfm} for the Macdonald polynomial $P_\mu(X;q,t)$ to $t=0$; this immediately leads to the formula (\ref{ryform}) below. In fact, we use the slight rephrasing of the Ram-Yip formula in terms of folding pairs, which was given in \cite{lencfm}[Section 2.2]; to be precise, the elements $(w,J)$ in ${\mathcal F}(\Gamma)$ index the monomial terms in the formula. After we set $t=0$, only the following subset of ${\mathcal F}(\Gamma)$ is relevant:
\begin{equation}\label{defbf}\overline{{\mathcal F}}(\Gamma):=\left\{(w,J)\in{\mathcal F}(\Gamma)\::\:\frac{1}{2}\left(\ell(w)+\ell(\endch(w,J))-\#J\right)+\sum_{j\in J^-}\langle\rho,\beta_j^\vee\rangle=0\right\}\,.\end{equation}
We also use the following notation related to folding pairs:
\[\level(w,J):=\sum_{j\in J^-}l_j\,.\]

\begin{theorem}\cite{raycfm} \label{hlpthm} We have
\begin{equation}\label{ryform}P_{\mu}(X;q,0)=
\sum_{(w,J)\in\overline{{\mathcal F}}(\Gamma)}q^{\level(w,J)}\,x^{\weight(w,J)}\,.\end{equation}
\end{theorem}

We will now rephrase the formula (\ref{ryform}) in terms of the {\em quantum Bruhat graph}. The latter is defined by adding to the Bruhat graph (\ref{brgr}) the following edges, called quantum edges, labeled by positive roots $\alpha$:
\begin{equation}\label{qbrgr}w\stackrel{\alpha}{\longrightarrow} ws_\alpha\;\;\;\;\mbox{if}\;\; \ell(w s_\alpha)=\ell(w)-2\langle\rho,\alpha^\vee\rangle+1\,. \end{equation}
The quantum Bruhat graph expresses the {\em Chevalley formula} (cf. \cite{chesdc}) in the quantum cohomology of a generalized flag variety $G/B$ \cite{fawqps}. 

The rephrasing mentioned above is based on the following description of $\overline{\mathcal F}(\Gamma)$.

\begin{proposition}\label{condqbg}
Let $(w,J)=(w,(j_1<j_2<\ldots<j_s))$. We have $(w,J)\in\overline{\mathcal F}(\Gamma)$ if and only if $\endch(w,J)=1$ (the identity) and $\pi(w,J)$ yields the following path in the corresponding {quantum Bruhat graph}, cf. {\rm (\ref{alcchain})}:
\[w=w_0\stackrel{\beta_{j_1}}{\longleftarrow}w_1\stackrel{\beta_{j_2}}{\longleftarrow}\ldots\stackrel{\beta_{j_s}}{\longleftarrow}w_s=1\,.\]
\end{proposition}

\begin{proof}
The condition for $(w,J)$ to be in $\overline{{\mathcal F}}(\Gamma)$, cf. (\ref{defbf}), can be rewritten as follows:
\[\sum_{i\::\:j_i\in J^+}\frac{1}{2}\left(\ell(w_{i-1})-\ell(w_i)-1\right)+\sum_{i\::\:j_i\in J^-}\frac{1}{2}\left(\ell(w_{i-1})-\ell(w_i)+2\langle\rho,\beta_{j_i}^\vee\rangle-1\right)+\ell(w_s)=0\,.\]
Note that all terms in this sum are non-negative, which implies that they are equal to 0. This fact immediately translates into the above conditions on $\endch(w,J)$ and $\pi(w,J)$, cf. (\ref{brgr}) and (\ref{qbrgr}).
\end{proof}

\section{The type $A$ setup}\label{secta}

\subsection{The root system} We start with the basic facts about the root system of type $A_{n-1}$. 
We can identify the space $\h_\R^*$ with the quotient $V:=\R^n/\R(1,\ldots,1)$,
where $\R(1,\ldots,1)$ denotes the subspace in $\R^n$ spanned 
by the vector $(1,\ldots,1)$.  
Let $\varepsilon_1,\ldots,\varepsilon_n\in V$ 
be the images of the coordinate vectors in $\R^n$.
The root system is 
$\Phi=\{\alpha_{ij}:=\varepsilon_i-\varepsilon_j \::\: i\ne j,\ 1\leq i,j\leq n\}$.
The simple roots are $\alpha_i=\alpha_{i,i+1}$, 
for $i=1,\ldots,n-1$.
The weight lattice is $\Lambda=\Z^n/\Z(1,\ldots,1)$. The fundamental weights are $\omega_i = \varepsilon_1+\ldots +\varepsilon_i$, 
for $i=1,\ldots,n-1$. 
A dominant weight $\mu=\mu_1\varepsilon_1+\ldots+\mu_{n-1}\varepsilon_{n-1}$ is identified with the partition $(\mu_{1}\geq \mu_{2}\geq \ldots \geq \mu_{n-1}\geq\mu_n=0)$ of length at most $n-1$; more generally, a weight is identified with a composition of length $n$. Note that $\rho=(n-1,n-2,\ldots,0)$. Considering the Young diagram of the dominant weight $\mu$ as a concatenation of columns, whose heights are $\mu_1',\mu_2',\ldots$, corresponds to expressing $\mu$ as $\omega_{\mu_1'}+\omega_{\mu_2'}+\ldots$ (as usual, $\mu'$ is the conjugate partition to $\mu$). We fix a dominant weight  $\mu$, that we use throughout Sections \ref{secta} and \ref{chrev}.

The Weyl group $W$ is the symmetric group $S_n$, which acts on $V$ by permuting the coordinates $\varepsilon_1,\ldots,\varepsilon_n$. Permutations $w\in S_n$ are written in one-line notation $w=w(1)\ldots w(n)$. For simplicity, we use the same notation $(i,j)$ with $1\le i<j\le n$ for the root $\alpha_{ij}$ and the reflection $s_{\alpha_{ij}}$, which is the transposition $t_{ij}$ of $i$ and $j$.  

\subsection{The specialization of the alcove model}\label{seta} We proved in \cite[Corollary 15.4]{lapawg} that, for any $k=1,\ldots,n-1$,
we have the following $\omega_k$-chain, denoted by $\Gamma(k)$:
\begin{equation}\label{omegakchain}\begin{array}{lllll}
(&\!\!\!\!(1,n),&(1,n-1),&\ldots,&(1,k+1)\,,\\
&\!\!\!\!(2,n),&(2,n-1),&\ldots,&(2,k+1)\,,\\
&&&\ldots\\
&\!\!\!\!(k,n),&(k,n-1),&\ldots,&(k,k+1)\,\,)\,.
\end{array}\end{equation}
Hence, we can construct a $\mu$-chain as a concatenation $\Gamma:=\Gamma^{\mu_1}\ldots\Gamma^1$, where $\Gamma^j=\Gamma(\mu'_j)$. This $\mu$-chain is fixed throughout Sections \ref{secta} and \ref{chrev}. Thus, we can replace the notation ${{\mathcal F}}(\Gamma)$ and $\overline{{\mathcal F}}(\Gamma)$ with ${{\mathcal F}}(\mu)$ and $\overline{{\mathcal F}}(\mu)$, respectively.

\begin{example}\label{ex21} {\rm Consider $n=4$ and $\mu =(3,2,1,0)$, for which we have the following $\mu$-chain (the underlined pairs are only relevant in Example \ref{ex21c} below):
\begin{equation}\label{exlchain}\Gamma=\Gamma^3\Gamma^2\Gamma^1=({(1,4)},(1,3),\underline{(1,2)}\:|\:(1,4),{(1,3)},\underline{(2,4)},\underline{(2,3)}\:|\:{(1,4)},\underline{(2,4)},\underline{(3,4)})\,.\end{equation}
Here the splitting of $\Gamma$ into $\Gamma^j$ is shown by bars. In order to visualize this $\mu$-chain, let us represent the Young diagram of $\mu$ inside a broken $3\times 4$ rectangle, with columns from shortest to longest (left to right), as shown below. In this way, a transposition $(i,j)$ in $\Gamma$ can be viewed as swapping entries in the two parts of each column (in rows $i$ and $j$, where the row numbers are also indicated below). 
\begin{equation*}
\begin{array}{l} \tableau{{1}&{1}&{1}\\ &{2}&{2}\\&&{3}}\\ \\
\tableau{{2}\\ {3}&{3}\\ {4}&{4}&{4}} \end{array}
\end{equation*}}
\end{example}

We will use positions in the $\mu$-chain $\Gamma$ and the corresponding roots interchangeably. For instance, we will replace the subsets $J=\{ j_1<\ldots< j_s\}$ in Section \ref{alcovewalks} with the corresponding sequences of roots, viewed as concatenations with distinguished factors $T=T^{\mu_1}\ldots T^1$ induced by the factorization of $\Gamma$ as $\Gamma^{\mu_1}\ldots\Gamma^1$, cf. Example \ref{ex21c}. As a consequence, we will refer to folding pairs as pairs $(w,T)$. 
We denote by $wT^{\mu_1}\ldots T^{j}$ the permutation obtained from $w$ via right multiplication by the transpositions in $T^{\mu_1},\ldots, T^{j}$, considered from left to right. This agrees with the above convention of using pairs to denote both roots and the corresponding reflections. As such, $\endch(w,T)$ can now be written simply $wT$. 

\begin{example}\label{ex21c}{\rm We continue Example \ref{ex21}, by picking the folding pair $(w,J)$ with $w=2134\in S_4$ and $J=\{3,6,7,9,10\}$ (see the underlined positions in (\ref{exlchain})). Thus, we have
\[T=T^3T^2T^1=((1,2)\:|\:(2,4),(2,3)\:|\:(2,4),(3,4))\,.\]
The corresponding chain in Bruhat order $\pi(w,T)$ is the following, where the swapped entries are shown in bold (we represent permutations as broken columns, as discussed in Example \ref{ex21}):
\[w=\begin{array}{l}\tableau{{{\mathbf 2}}} \\ \\ \tableau{{\mathbf 1}\\{3}\\{{4}}} \end{array} \! >\! \begin{array}{l}\tableau{{ 1}} \\ \\ \tableau{{ 2}\\{3}\\{4}} \end{array}\:|\: \begin{array}{l}\tableau{{{ 1}}\\{{\mathbf 2}}} \\ \\ \tableau{{{ 3}}\\{\mathbf 4}} \end{array}\!{<}\! \begin{array}{l}\tableau{{{ 1}}\\{{\mathbf 4}}} \\ \\ \tableau{{{\mathbf 3}}\\{ 2}}\end{array}\!>\!  \begin{array}{l}\tableau{{1}\\{ 3}}\\ \\ \tableau{{ 4}\\{2}}\end{array}\:|\: 
\begin{array}{l}\tableau{{1}\\{{\mathbf 3}}\\{4}}\\ \\ \tableau{{{\mathbf 2}}}\end{array} 
\!>\!   \begin{array}{l} \tableau{{1}\\{ 2}\\{\mathbf 4}} \\ \\ \tableau{{\mathbf 3}}\end{array} \! >\!   \begin{array}{l} \tableau{{1}\\{2}\\{ 3}} \\ \\ \tableau{{ 4}} \end{array} \,.\]
We conclude that $(w,T)$ belongs to $\overline{{\mathcal F}}(\mu)$, and $J^+=\{3,7,9,10\}$, $J^-=\{6\}$. 
}
\end{example}

\subsection{The filling map}\label{setafill} We will now associate with a folding pair $(w,T)$ a {filling} of the Young diagram $\mu$, which is viewed as a concatenation of columns $\sigma=C^{\mu_1}\ldots C^1$ of heights $\ldots,\mu_2',\mu_1'$ (from left to right). The cells in each column are filled with distinct numbers from $1$ to $n$, with no restriction on their order.  

Given a folding pair $(w,T)$, we consider the permutations
\[\pi^j:=wT^{\mu_1}T^{\mu_1-1}\ldots T^{j+1}\,,\]
for $j=0,\ldots,\mu_1$. In particular, $\pi^{\mu_1}=w$ and, if $(w,T)\in \overline{\mathcal F}(\mu)$, then $\pi^0$ is the identity. Given a permutation $u$, we also use the notation $u[i,j]:=u_i\ldots u_j$.

\begin{definition}\label{deffill}
The {\em filling map} is the map $f$ from folding pairs $(w,T)$ in ${\mathcal F}(\mu)$ to fillings $f(w,T)=C^{\mu_1}\ldots C^1$ of the shape $\mu$ defined by
\begin{equation}\label{defswt}C^j:=\pi^j[1,\mu_j']\,,\;\;\;\;\mbox{for $j=1,\ldots,\mu_1$}\,.\end{equation}
\end{definition}

\begin{example}\label{exfill} {\rm Given $(w,T)$ as in Example \ref{ex21c}, we have
\[f(w,T)=\tableau{{2}&{1}&{1}\\&{2}&{3}\\&&{4}}\,.\]}
\end{example}

As usual, we define the content of a filling $\sigma$ as $\content(\sigma):=(c_1,\ldots,c_n)$, where $c_i$ is the number of entries $i$ in the filling. We also let $x^{\content(\sigma)}:=x_1^{c_1}\ldots x_n^{c_n}$. We recall the following result in \cite{lenhlp}[Proposition 3.6].

\begin{proposition}\label{weightmon}\cite{lenhlp} Given an arbitrary folding pair $(w,T)$ in ${\mathcal F}(\mu)$, we have \linebreak $\content(f(w,T))=\weight(w,T)$. In particular, $\weight(w,T)$ only depends on $f(w,T)$. \end{proposition}

\subsection{The quantum Bruhat graph}\label{setaqbr} In this section we present a criterion for the edges of the corresponding quantum Bruhat graph. 
We need the circular order $\prec_i$ on $[n]$ starting at $i$, namely $i\prec_i i+1\prec_i\ldots \prec_i n\prec_i 1\prec_i\ldots\prec_i i-1$. It is convenient to think of this order in terms of the numbers $1,\ldots,n$ arranged on a circle clockwise. We make the convention that, whenever we write $a\prec b\prec c\prec\ldots$, we refer to the circular order $\prec=\prec_a$.

\begin{proposition}\label{crita} We have an edge $w\stackrel{(i,j)}{\longrightarrow} w(i,j)$ if and only if there is no $k$ such that $i<k<j$ and $w(i)\prec w(k)\prec w(j)$.
\end{proposition}

\begin{proof}
Given a permutation $w$ in $S_n$ and $1\le i<j\le n$ such that $a:=w(i)<b:=w(j)$, the realization of the length in terms of inversions implies
\[\ell(w(i,j))-\ell(w)=2N_{ab}(w[i,j])+1\,;\] 
here $N_{ab}(u)$ denotes the number of letters of the word $u$ which are strictly between $a$ and $b$. The criterion follows immediately from this fact. 
\end{proof}

\begin{remark}\label{across} If we apply to $w$ a transposition $(i,j)$ such that $i<k<j$ and $w(i)\prec w(k)\prec w(j)$ for some $k$, we say that we transpose the values $w(i)$ and $w(j)$ across $w(k)$. 
\end{remark}

\section{The type $A$ charge revisited}\label{chrev}

\subsection{The main construction}\label{constra} Let $B_\mu$ denote the set of column-strict fillings of the shape $\mu$ with integers in $[n]$. We  write
\[B_\mu=\bigotimes_{i=\mu_1}^{1} B^{\mu_i',1}\,;\]
here $B^{k,1}$ is the traditional notation for the type $A_{n-1}$ {\em Kirillov-Reshetikhin crystal} \cite{karrym} indexed by a column of height $k$, whose vertices are indexed by increasing fillings of the mentioned column with integers in $[n]$. 

Consider the following composite map:
\begin{equation}\label{compa}\overline{\mathcal F}(\mu)\stackrel{f}{\longrightarrow} f(\overline{\mathcal F}(\mu))\stackrel{{\rm ord}}{\longrightarrow} B_\mu\,,\end{equation}
where $f$ is the {filling map} in Definition \ref{deffill}, and the map ``${\rm ord}$'' is sorting each column increasingly. The following is the main result in type $A$, and this section is devoted to its proof.

\begin{theorem}\label{bija} The composite ${\rm ord}\circ f$ is a bijection between $\overline{\mathcal F}(\mu)$ and $B_\mu$. \end{theorem}

\begin{remarks}\label{remfilla} (1) In a later publication, we will show how to define the crystal operators on $\overline{\mathcal F}(\mu)$ such that the above bijection is a crystal isomorphism. This definition will generalize the one in the alcove model \cite{lapcmc}, and will have several applications.

(2) Let us restrict the map $f$ to those $(w,J)$ in $\overline{\mathcal F}(\mu)$ for which the path associated to $\pi(w,J)$, cf. Proposition \ref{condqbg}, is a path in the Bruhat graph. The increasing length condition on the path ensures that the rows of $f(w,J)$ are weakly increasing (from right to left, in our display); moreover, the saturated chain condition ensures that the columns of $f(w,J)$ are strictly increasing. Thus, the mentioned restriction of $f$ is a bijection to semistandard Young tableaux of shape $\mu$ with entries in $[n]$; these index the vertices of the crystal $B(\mu)$ of highest weight $\mu$, which is embedded in $B_\mu$. 
\end{remarks}

Consider the following two conditions on a pair of adjacent columns $C'C$ (cf. the notation in Section \ref{setafill}) in a filling of a Young diagram. 

\begin{condition}\label{cond1} For any pair of indices $1\le i<l\le \#C'$, both statements below are false: 
\begin{equation}\label{cond12}C(i)=C'(l)\,,\;\;\;\;\;\;\;\;C(i)\prec C'(l)\prec C'(i)\,.\end{equation}
\end{condition}

\begin{condition}\label{cond2}  For every index $1\le i\le \#C'$, we have
\[C'(i)=\min\,\{C'(l)\::\:i\le l\le\#C'\}\,,\]
where the minimum is taken with respect to the circular order $\prec_{C(i)}$ on $[n]$ starting at $C(i)$.
\end{condition}

\begin{remark}\label{remcond} It is clear that the two conditions are equivalent. We will show in the proof of Lemma \ref{chconst} that they are closely related to the quantum Bruhat graph. It is straightforward that, given any filling $\tau$ in $B_\mu$, there is a unique filling $\sigma$ satisfying the following conditions: (i) the first column of $\sigma$ (of height $\mu_1'$) is increasing, (ii) the adjacent columns of $\sigma$ satisfy Condition \ref{cond1}, and (iii) ${\rm ord}(\sigma)=\tau$. Indeed, the columns of $\sigma$ can be constructed from longest to shortest (right to left, cf. the notation in Section \ref{setafill}) and from top to bottom by repeatedly using Condition \ref{cond2}. This algorithm can be traced back (in a row version, as opposed to the present column version) to \cite{hhlcfm}[Section 7], and is illustrated by Example \ref{exrinv} below. Finally, note that if $C^0=12\ldots n$, then Condition \ref{cond2} on $CC^0$ is equivalent to $C$ being increasing. This fact is used in several instances below, where we consider the augmentation of a filling $\sigma=\ldots C^2C^1$ to the filling $\widehat{\sigma}=\ldots C^2C^1C^0$. 
\end{remark}

\begin{example}\label{exrinv}{\rm 
The algorithm described in Remark {\rm \ref{remcond}} easily leads from the filling $\tau$ below to $\sigma$. The bold entries in $\sigma$ are only relevant in Example {\rm \ref{exch}} below.
\begin{equation}\label{st}\tau=\tableau{{2}&{1}&{2}&{3}\\&{2}&{3}&{5}\\&{4}&{4}&{6}}\,,\;\;\;\;\;\sigma=\tableau{{2}&{\mathbf 4}&{3}&{3}\\&{2}&{2}&{\mathbf 5}\\&{1}&{\mathbf 4}&{\mathbf 6}}\,.\end{equation}
}
\end{example}

The proof of Theorem \ref{bija} consists essentially of the following result.

\begin{proposition}\label{chaina}
The restriction of the filling map $f$ to $\overline{\mathcal F}(\mu)$ is injective. The image of this restriction consists of all fillings of $\mu$ with integers in $[n]$ whose first column is increasing and whose adjacent columns satisfy Condition {\rm \ref{cond1}}. 
\end{proposition}

The proof of Proposition \ref{chaina} is based on the following result.

\begin{lemma}\label{chconst} {\rm (1)} Let $u$ be a permutation in $S_n$, and $C$ a filling of a column of height $k$ with distinct entries in $[n]$, such that $C[i+1,k]=u[i+1,k]$ for a fixed $i\le k$. Assume that there is a sequence $k<m_1<\ldots<m_p\le n$ such that we have a path in the quantum Bruhat graph starting at $u$ with edges labeled $(i,m_1),\ldots,(i,m_p)$, in this order, where $u(m_p)=C(i)$.  Then the following two conditions are satisfied: $C(i)\ne u(l)$ for $l<i$, and $u(i)\prec C(l)\prec C(i)$ fails for $i<l\le k$.  

{\rm (2)} Viceversa, if $u(i)\ne C(i)$, the two conditions above imply the existence of a unique sequence $(m_1,\ldots,m_p)$ with the above quantum Bruhat graph property. Moreover, we have
\begin{equation}\label{monota}
u(i)\prec u(m_1)\prec\ldots\prec u(m_p)=C(i)\,.
\end{equation}
\end{lemma}

We claim that the construction of the sequence $S=((i,m_1),\ldots,(i,m_p))$ in Lemma \ref{chconst} (2) is given by the following greedy algorithm. The procedure {\em path-A} is called with the parameters $u$, $i$, and $c=C(i)$, while $L$ is the list of positions $(k+1,\ldots,n)$. The function {\em next(m,L)} determines the succesor of the element $m$ in the list $L$. Note that the algorithm works even when $u(i)= C(i)$, which will be needed later.

\begin{algorithm}\label{algchain}\hfill \\
procedure path-A$(u,i,c,L)$;\\
if $u(i)=c$ then return $\emptyset,u$\\
else\\
\indent let $S:=\emptyset$, $m:=L(1)$, $v:=u$;\\
\indent while $v(m)\ne c$ do \\
\indent \indent if $v(i)\prec v(m)\prec c$ then let $S:=S,(i,m)$, $v:=v(i,m)$;\\
\indent \indent  end if;\\
\indent \indent  let $m:=next(m,L)$;\\
\indent end while;\\
\indent let $S:=S,(i,m)$, $v:=v(i,m)$;\\
\indent return $(S,v)$;\\
end if;\\
end.
\end{algorithm}

\begin{proof}[Proof of Lemma {\rm \ref{chconst}}] (1) The condition $C(i)\ne u(l)$ for $l<i$ is trivial. On the other hand, if $u[i,k]$ contains an entry $a$ with $u(i)\prec a\prec C(i)$, then one of the edges of the considered path in the quantum Bruhat graph would violate the criterion in Proposition \ref{crita}.

(2) We prove the above claim related to Algorithm \ref{algchain}. The fact that $C(i)\ne u(l)$ for $l\le i$ implies that $C(i)$ appears in a position $q>k$ in $u$. This ensures that the procedure {\em path-A$(u,i,C(i),(k+1,\ldots,n))$} terminates (correctly). This procedure constructs a sequence $k<m_1<\ldots<m_p=q\le n$ such that (\ref{monota}) holds and, for each $l=1,\ldots,p-1$, the sequence $u[m_l,m_{l+1}-1]$ contains no entry $a$ with $u(m_l)\prec a\prec C(i)$. In fact, the latter condition also holds for $l=0$ if we set $m_0:=i$, by the hypothesis.  Thus the output of the procedure has the desired quantum Bruhat graph property, by the corresponding criterion. On the other hand, no other sequence $k<m_1'<\ldots<m_r'=q$ has the same property. Indeed, if $m_1'>m_1$ (respectively $m_1'<m_1$), then at some point in the process of applying the transpositions $(i,m_l')$ to $u$, the entry in position $i$ has to change from a value $a$ to a value $b$ across the value $u(m_1)$ (respectively $u(i)$), see Remark \ref{across}; but this violates the quantum Bruhat graph criterion. Therefore $m_1=m_1'$, and the reasoning continues in the same way. 
\end{proof}

\begin{proof}[Proof of Proposition {\rm \ref{chaina}}] We will implicitly use the part of Remark \ref{remcond} related to an augmented filling of the shape $\mu$. By Lemma \ref{chconst} (1), any filling in the image of $\overline{\mathcal F}(\mu)$ under $f$ has the stated properties. On the other hand, given a filling $\sigma=C^{\mu_1}\ldots C^1$ with the stated properties, we can show that there is a unique folding pair $(w,T)$ mapped to it, based on Lemma \ref{chconst} (2). More precisely, the reverse of the sequence $T$ is determined by succesive calls of the procedure {\em path-A$(u,i,c,L)$}. Here $c=C^j(i)$ and $L=(\#C^j+1,\ldots,n)$, where $j$ ranges from $1$ to $\mu_1$, while for each $j$ the index $i$ ranges from $\#C^j$ down to $1$; the permutation $u$ starts by being the identity, and then is set to the permutation $v$ returned by the previous call of the procedure. Finally, $T$ is set to the reverse of the (left to right) concatenation of the outputs $S$ in the successive calls of the procedure, while $w$ is set to permutation returned by the last call. The above algorithm reconstructs, for instance, the path in the quantum Bruhat graph in Example \ref{ex21c} from the filling in Example \ref{exfill}.
\end{proof}

\begin{proof}[Proof of Theorem {\rm \ref{bija}}] This is now immediate, based on Remark \ref{remcond} and Proposition \ref{chaina}.
\end{proof}

\subsection{Rederiving the classical charge}\label{redcha} We start by recalling the construction of the classical charge of a word due to Lascoux and Sch\"utzenberger \cite{lassuc}. Assume that $w$ is a word with letters in the alphabet $[n]$ which has partition content, i.e., the number of $j$'s is at least the number of $j+1$'s, for each $j=1,\ldots,n-1$. The statistic $\charge(w)$ is calculated as a sum based on the following algorithm. Scan the  word starting from its right end, and select the numbers $1,2,\ldots$ in this order, up to the largest possible $k$. We always pick the first available entry $j+1$ to the left of the previous entry $j$. Whenever there is no such entry, we pick the rightmost entry $j+1$, so we start scanning the word from its right end once again; in this case, we also add $k-j$ to the sum that computes $\charge(w)$. At the end of this process, we remove the selected numbers and repeat the whole procedure until the word becomes empty. 

\begin{example}\label{lsex}{\rm Consider the word $w=11{\mathbf 3}2\mathbf{214}323$, where the first group of selected numbers is shown in bold. The corresponding contribution to the charge is $1$. After removing the bold numbers and another round of selections (again shown in bold), we have $1{\mathbf 1}2\mathbf{32}3$, so the contribution to the charge is $2$. We are left with the word $123$, whose contribution to the charge is $2+1=3$. So $\charge(w)=1+2+(2+1)=6$. }
\end{example}

We now reinterpret the charge from the point of view of the Ram-Yip formula (\ref{ryform}) and Theorem \ref{bija}. Given a filling $\tau$ in $B_\mu$ (i.e., a column-strict filling of $\mu$), we define its {\em charge word} as the biword $\cw(\tau)$ containing a biletter $\binom{k}{j}$ for each entry $k$ in column $j$ of $\tau$ (the columns are numbered as usual, from longest to shortest). We order the biletters in the decreasing order of the $k$'s, and for equal $k$'s, in the decreasing order of $j$'s. The obtained word formed by the lower letters $j$ will be denoted by $\cw_2(\tau)$. We refer to Example \ref{exch} for an illustration of the charge word.

Given an arbitrary filling $\sigma=C^{\mu_1}\ldots C^1$, we say that the cell in column $C^j$ and row $i$ is a descent if $C^j(i)>C^{j+1}(i)$, assuming that $C^{j+1}(i)$ is defined. Let ${\rm Des}(\sigma)$ denote the set of descents in $\sigma$. As usual, we define the arm length ${\rm arm}(c)$ of a cell $c$ as the number of cells to its left (columns are arranged from shortest to longest, left to right, as in Section \ref{setafill}).

Recall the bijections in (\ref{compa}), and compose also with the map $\cw_2$, obtaining
\begin{equation}\label{mapa}(w,T)\mapsto \sigma=f(w,T)\mapsto \tau={\rm ord}(\sigma)\mapsto \cw_2(\tau)\,.\end{equation}
We will now show that the algorithm in Remark \ref{remcond} for reconstructing $\sigma$ from $\tau$ translates, in the setup of the charge word, precisely into the selection algorithm which computes the charge of $\cw_2(\tau)$ (see the second part of the proof of Theorem \ref{levch} below). Since the former algorithm is closely related to the quantum Bruhat graph, as we discussed in Section \ref{constra}, we can conclude that this graph explains the charge construction itself. Furthermore, we have the following translation of statistics via the maps in (\ref{mapa}).

\begin{theorem}\label{levch} We have
\[\level(w,T)=\sum_{c\in{\rm Des}(\sigma)}{\rm arm}(c)=\charge(\cw_2(\tau))\,.\]
\end{theorem}

\begin{proof}
Consider two consecutive columns $C^{j+1}C^j$ of the augmented filling $\widehat{\sigma}$, cf. Remark \ref{remcond}. Fix $i$ with $1\le i\le k:=\#C^{j+1}$, and assume that the passage from $\pi^{j}$ to $\pi^{j+1}$ involves transposing position $i$ with positions $k<m_1<\ldots<m_p$. In this process, the $i$th entry of the corresponding permutation takes successive values $a_0=C^j(i)\prec a_1\prec\ldots\prec a_p=C^{j+1}(i)$ in this order, by (\ref{monota}). Therefore, the sequence $(i,m_1), \ldots,(i,m_p)$ contains a negative folding $(i,m_t)$, in fact a unique one, if and only if $C_j(i)>C_{j+1}(i)$. In this case, if we let $s$ be the position of the mentioned root $(i,m_t)$ in the fixed $\mu$-chain $\Gamma$, the corresponding affine level $l_s$ (see Section \ref{alcovewalks}) is, by definition, the number of occurences of $(i,m_t)$ in the segments $\Gamma^{j+1},\ldots,\Gamma^{\mu_1}$ of $\Gamma$. But this is precisely the arm length of the $i$th cell in $C^j$, by the form of $\Gamma$. The expression of $\level(w,T)$ in terms of the descents of $\sigma$ now follows. 

For the second equality, just note that the criterion for constructing the sequences $1,2,\ldots$ in the definition of $\charge(\cw_2(\tau))$ mirrors Condition \ref{cond2}, which is used to reconstruct $\sigma$ from $\tau$ (cf. the algorithm in Remark \ref{remcond}). More precisely, consider the $i$th sequence $1,2,\ldots$ extracted from $\cw_2(\tau)$ (which turns out to have length $\mu_i$), and the letter $j$ in this sequence; then the top letter paired with the mentioned letter $j$ in $\cw(\tau)$ is precisely the entry in row $i$ and column $j$ of the filling $\sigma$. In particular, the steps to the right in the $i$th iteration of the charge computation correspond precisely to the descents in the $i$th row of $\sigma$, while the corresponding charge contributions and arm lengths coincide. 
\end{proof}

\begin{example}\label{exch}{\rm Note that $\cw_2(\tau)$ for $\tau$ in Example \ref{exrinv} is precisely the word $w$ in Example \ref{lsex}. In fact, the full biword $\cw(\tau)$ is shown below, using the order on the biletters specified above. The index attached to a lower letter is the number of the iteration in which the given letter is selected in the process of computing $\charge(\cw_2(\tau))$. 
\[\cw(\tau)=\left(\begin{array}{cccccccccc}6&5&4&4&3&3&2&2&2&1\\1_3&1_2&3_1&2_3&2_1&1_1&4_1&3_2&2_2&3_3\end{array}\right)\,.\]
One can note the parallel between the mentioned selection process and the reconstruction of the filling $\sigma$ from $\tau$, cf. the proof of Theorem \ref{levch} and Example \ref{exrinv}. The entries in the cells of ${\rm Des}(\sigma)$ are shown in bold in (\ref{st}).
}
\end{example}

Let us replace, for simplicity, the notation $\charge(\cw_2(\tau))$ by $\charge(\tau)$. We obtain the following corollary of Theorem \ref{levch} and the Ram-Yip formula (\ref{ryform}), based on Proposition \ref{weightmon}.

\begin{corollary}
In type $A$, we have
\[P_{\mu}(X;q,0)=
\sum_{\tau\in B_\mu}q^{\charge(\tau)}\,x^{\content(\tau)}\,.\]
\end{corollary}

\section{The type $C$ setup}\label{sectc}

\subsection{The root system} We start with the basic facts about the root system of type $C_{n}$. We can identify the space $\h_\R^*$ with  $V:=\R^n$, the coordinate vectors being $\varepsilon_1,\ldots,\varepsilon_n$.  
The root system is 
$\Phi=\{\pm\varepsilon_i\pm\varepsilon_j \::\:  1\leq i<j\leq n\}\cup\{\pm 2\varepsilon_i\::\: 1\leq i\leq n\}$. 
The simple roots are $\alpha_i=\varepsilon_i-\varepsilon_{i+1}$, 
for $i=1,\ldots,n-1$, and $\alpha_n=2\varepsilon_n$. 
The weight lattice is $\Lambda=\Z^n$. The fundamental weights are $\omega_i = \varepsilon_1+\ldots +\varepsilon_i$, 
for $i=1,\ldots,n$. 
A dominant weight $\mu=\mu_1\varepsilon_1+\ldots+\mu_n\varepsilon_n$ is identified with the partition $(\mu _{1}\geq \mu _{2}\geq \ldots \geq \mu _{n-1}\geq\mu_n\geq 0)$ of length at most $n$. Note that $\rho=(n,n-1,\ldots,1)$. Like in type $A$, writing the dominant weight $\mu$ as a sum of fundamental weights corresponds to considering the Young diagram of $\mu$ as a concatenation of columns. We fix a dominant weight  $\mu$, that we use throughout Sections \ref{sectc} and \ref{chc}.

The Weyl group $W$ is the group of signed permutations $B_n$, which acts on $V$ by permuting the coordinates and changing their signs. A signed permutation is a bijection $w$ from $[\overline{n}]:=\{1<2<\ldots<n<\overline{n}<\overline{n-1}<\ldots<\overline{1}\}$ to $[\overline{n}]$ satisfying $w(\overline{\imath})=\overline{w(i)}$. Here $\overline{\imath}$ is viewed as $-i$, so $\overline{\overline{\imath}}=i$, $|\overline{\imath}|=i$, and $\sign(\overline{\imath})=-1$. We use both the window notation $w=w(1)\ldots w(n)$ and the full one-line notation $w=w(1)\ldots w(n)w(\overline{n})\ldots w(\overline{1})$ for signed permutations. For simplicity, given $1\le i<j\le n$, we denote by $(i,j)$ the root $\varepsilon_i-\varepsilon_j$ and the corresponding reflection, which is identified with the composition of transpositions $t_{ij}t_{\overline{\jmath}\overline{\imath}}$. Similarly, we denote by $(i,\overline{\jmath})$, again for $1\le i<j\le n$, the root $\varepsilon_i+\varepsilon_j$ and the corresponding reflection, which is identified with the composition of transpositions $t_{i\overline{\jmath}}t_{j\overline{\imath}}$. Finally, we denote by $(i,\overline{\imath})$ the root $2\varepsilon_i$ and the corresponding reflection, which is identified with the transposition $t_{i\overline{\imath}}$. The length of an element $w$ in $B_n$ is given by
\begin{equation}\label{lengthb}\ell(w):=\#\{(k,l)\in[n]\times[\overline{n}]\::\: k\le|l|,\:w(k)>w(l)\}\,.\end{equation}

\subsection{The specialization of the alcove model}\label{setc} Consider the sequence of roots $\Gamma(k):=\Gamma_r(k)\Gamma_l(k)$, where
\begin{align*}
&\;\:\Gamma_r(k):=\Gamma_2\ldots\Gamma_k\,,\;\;\;\;\Gamma_l(k):=\Gamma_{k1}\ldots\Gamma_{kk}\,,\\
&\;\:\Gamma_i:=((1,\overline{\imath}),(2,\overline{\imath}),\ldots,(i-1,\overline{\imath}))\,,\\
&\begin{array}{llllll}\Gamma_{ki}:=(\!\!\!\!\!&(1,\overline{\imath}),&(2,\overline{\imath}),&\ldots,&(i-1,\overline{\imath}),\\
&(i,\overline{k+1}),&(i,\overline{k+2}),&\ldots,&(i,\overline{n}),\\&(i,\overline{\imath}),\\ 
&(i,n),&(i,n-1),&\ldots,&(i,k+1)\,)\,.\end{array}\end{align*}
We proved in \cite{lenhhl}[Lemma 4.1] that $\Gamma(k)$ is an $\omega_k$-chain.
Hence, we can construct a $\mu$-chain as a concatenation $\Gamma:=\Gamma^{\mu_1}\ldots\Gamma^1$, where $\Gamma^j=\Gamma(\mu'_j)$; we also let $\Gamma^{j}_r:=\Gamma_r(\mu_j')$ and $\Gamma^{j}_l:=\Gamma_l(\mu_j')$. 
This $\mu$-chain is fixed throughout Sections \ref{sectc} and \ref{chc}. Thus, we can replace the notation ${{\mathcal F}}(\Gamma)$ and $\overline{{\mathcal F}}(\Gamma)$ with ${{\mathcal F}}(\mu)$ and $\overline{{\mathcal F}}(\mu)$, respectively.

\begin{remark} The splitting of $\Gamma(k)$ as $\Gamma_r(k)\Gamma_l(k)$ corresponds to the splitting of a Kashiwara-Nakashima column into a right and a left column, cf. Section \ref{sectkn}; this explains the significance of the two indices.
\end{remark}

\begin{example}\label{ex31} {\rm Consider $n=3$ and $\mu =(2,1,0)$, for which we have the following $\mu$-chain $\Gamma=\Gamma^2\Gamma^1$ (the underlined pairs are only relevant in Example \ref{ex31c} below):
\begin{equation}\label{exlchainc}(\:|\:(1,\overline{2}),(1,\overline{3}),\underline{(1,\overline{1})},(1,3),\underline{(1,2)}\:||\:\underline{(1,\overline{2})}\:|\:(1,\overline{3}),(1,\overline{1}),(1,3),(1,\overline{2}),\underline{(2,\overline{3})},\underline{(2,\overline{2})},\underline{(2,3)})\,.\end{equation}
Here the splitting of $\Gamma$ into $\Gamma^j$ is shown by double bars, while the splitting of each $\Gamma^j$ as $\Gamma^{j}_r\Gamma^{j}_l$ is shown by single bars.
}
\end{example}

Like in type $A$, we will use positions in the $\mu$-chain $\Gamma$ and the corresponding roots interchangeably. For instance, we replace the subset $J$ in a folding pair $(w,J)$ with the subsequence $T$ of the $\mu$-chain $\Gamma$ indexed by the positions in $J$. The factorizations of $\Gamma$ with factors $\Gamma^j$ and $\Gamma^{j}_r,\,\Gamma^{j}_l$ induces factorizations of $T$ with factors $T^j$ and $T^{j}_r,\,T^{j}_l$, respectively. 

\begin{example}\label{ex31c}{\rm We continue Example \ref{ex31}, by picking the folding pair $(w,J)$ with $w=123\in B_3$ and $J=\{3,5,6,11,12,13\}$ (see the underlined positions in (\ref{exlchainc})). Thus, we have
\[T=T^2T^1=(\:|\:(1,\overline{1}),(1,2)\:||\:(1,\overline{2})\:|\:(2, \overline{3}),(2, \overline{2}),(2,3))\,.\]
The corresponding chain in Bruhat order $\pi(w,T)$ is the following, where the swapped entries are shown in bold (we represent signed permutations in the window notation as broken columns, as discussed in Example \ref{ex21}):
\[  w=\begin{array}{l}\tableau{{1}}\\ \\ \tableau{{2}\\{3}} \end{array}|
\begin{array}{l}\tableau{{{{\mathbf 1}}}}\\ \\ \tableau{{{2}}\\{{{3}}}} \end{array} \!{<}\! \begin{array}{l} \tableau{{\mathbf {\overline{1}}}}\\ \\ \tableau{{{\mathbf 2}}\\{{{3}}}} \end{array}\!>\!
\begin{array}{l} \tableau{{ 2}}\\ \\ \tableau{{{\overline{1}}}\\{3}} \end{array}||
\begin{array}{l} \tableau{{{\mathbf 2}}\\{\mathbf {\overline{1}}}}\\ \\ \tableau{{3}} \end{array}\!>\!
\begin{array}{l} \tableau{{ 1}\\{{\overline{2}}}}\\ \\ \tableau{{3}} \end{array}|
\begin{array}{l} \tableau{{1}\\{\mathbf {\overline{2}}}}\\ \\ \tableau{{\mathbf 3}} \end{array}\!>\!\begin{array}{l} \tableau{{1}\\{\mathbf {\overline{3}}}}\\ \\ \tableau{{{ 2}}} \end{array}\!>\!
\begin{array}{l} \tableau{{1}\\{{\mathbf 3}}}\\ \\ \tableau{{{\mathbf 2}}}\end{array}\!>\!\begin{array}{l} \tableau{{1}\\{{ 2}}}\\ \\ \tableau{{{ 3}}} \end{array}.
\]
We conclude that $(w,T)$ belongs to $\overline{{\mathcal F}}(\mu)$, and $J^+=\{5,6,11,12,13\}$, $J^-=\{3\}$. 
}
\end{example}

\subsection{The filling map}\label{setcfill} We will now associate with a folding pair $(w,T)$ a {filling} of the Young diagram $2\mu$, which is viewed as a concatenation of columns $\sigma=C^{\mu_1}_rC^{\mu_1}_l\ldots C^{1}_rC^{1}_l$, where $C^{j}_r$ and $C^{j}_l$ have height $\mu_j'$. The cells in each column are filled with numbers in $[\overline{n}]$ with distinct absolute values, with no restriction on their order.  

Given a folding pair $(w,T)$, we consider the signed permutations
\[\pi^{j}_r:=wT^{\mu_1}T^{\mu_1-1}\ldots T^{j+1}\,,\;\;\;\;\pi^{j}_l:=\pi^{j}_rT^{j}_r\,,\]
for $j=1,\ldots,\mu_1$; here and in the definition below we use the same notation as in Sections \ref{seta} and \ref{setafill}. 

\begin{definition}\label{deffillc}
The {\em filling map} is the map $f$ from folding pairs $(w,T)$ in ${\mathcal F}(\mu)$ to fillings $f(w,T)=C^{\mu_1}_rC^{\mu_1}_l\ldots C^{1}_rC^{1}_l$ of the shape $2\mu$ defined by
\begin{equation}\label{defswtc}C^{j}_r:=\pi^{j}_r[1,\mu_j']\,,\;\;\;C^{j}_l:=\pi^{j}_l[1,\mu_j']\,,\;\;\;\;\;\;\mbox{for $j=1,\ldots,\mu_1$}.\end{equation}
\end{definition}

\begin{example} {\rm Given $(w,T)$ as in Example \ref{ex31c}, we have
\[f(w,T)=\tableau{{{1}}&{{1}}&{2}&{1}\\&&{\overline{1}}&{\overline{2}}}\,.\]}
\end{example}

As usual, we define the content of a filling $\sigma$ as $\content(\sigma):=(c_1,\ldots,c_n)$, where $c_i$ is half the difference
between the number of occurences of the entries $i$ and $\overline{\imath}$ in the filling $\sigma$. We also let $x^{\content(\sigma)}:=x_1^{c_1}\ldots x_n^{c_n}$. We recall the following result in \cite{lenhhl}[Theorem 4.6 (2)].

\begin{proposition}\label{weightmonc}\cite{lenhhl} Given an arbitrary folding pair $(w,T)$ in ${\mathcal F}(\mu)$, we have \linebreak $\content(f(w,T))=\weight(w,T)$. In particular, $\weight(w,T)$ only depends on $f(w,T)$. \end{proposition}

\subsection{The quantum Bruhat graph}\label{setcqbr} In this section we present a criterion for the edges of the corresponding quantum Bruhat graph. 
The circular order $\prec_i$ on $[\overline{n}]$ starting at $i$ in $[\overline{n}]$ is defined in the obvious way, cf. Section \ref{setaqbr}. It is convenient to think of this order in terms of the numbers $1,\ldots,n,\overline{n},\ldots,\overline{1}$ arranged on a circle clockwise. We make the same convention as in Section \ref{setaqbr} related to a chain of inequalities $a\prec b\prec c\prec\ldots\,$.

\begin{proposition}\label{critc} {\rm (1)} Given $1\le i<j\le n$, we have an edge $w\stackrel{(i,j)}{\longrightarrow} w(i,j)$ if and only if there is no $k$ such that $i<k<j$ and $w(i)\prec w(k)\prec w(j)$. 

{\rm (2)} Given $1\le i<j\le n$, we have an edge $w\stackrel{(i,\overline{\jmath})}{\longrightarrow} w(i,\overline{\jmath})$ if and only if $w(i)<w(\overline{\jmath})$, $\sign(w(i))=\sign(w(\overline{\jmath}))$, and there is no $k$ such that $i<k<\overline{\jmath}$ and $w(i)< w(k)< w(\overline{\jmath})$. 

{\rm (3)} Given $1\le i\le n$, we have an edge $w\stackrel{(i,\overline{\imath})}{\longrightarrow} w(i,\overline{\imath})$ if and only if there is no $k$ such that $i<k<\overline{\imath}$ (or, equivalently, $i<k\le n$) and $w(i)\prec w(k)\prec w(\overline{\imath})$.
\end{proposition}

\begin{proof}
Recall the notation $N_{ab}(u)$ in the proof of Proposition \ref{crita}. The proof is based on the facts below related to an element $w$ in $B_n$, which follow easily from (\ref{lengthb}):
\begin{enumerate}
\item given $1\le i<j\le n$ such that $a:=w(i)<b:=w(j)$, we have
\[\ell(w(i,j))-\ell(w)=2N_{ab}(w[i,j])+1\,;\] 
\item given $1\le i<j\le n$ such that $a:=w(i)<b:=w(\overline{\jmath})$, we have
\[\ell(w(i,\overline{\jmath}))-\ell(w)=2N_{ab}(w[i,j-1])+2N_{ab}(w[j+1,\overline{\jmath}])+2\delta_{\sign(a),-\sign(b)}+1\,,\]
where $\delta_{k,l}$ is the Kronecker delta;
\item given $1\le i\le n$ such that $a:=w(i)\in[n]$, we have
\[\ell(w(i,\overline{\imath}))-\ell(w)=2N_{a\overline{a}}(w[i,n])+1\,.\] 
\end{enumerate}
Note that we can never have an edge $w\stackrel{(i,\overline{\jmath})}{\longleftarrow}w(i,\overline{\jmath})$ in the quantum Bruhat graph in case (2) because
\[2\langle\rho,\varepsilon_i+\varepsilon_j\rangle-1=2(2n-i-j)+3>2(2n-i-j)+1\ge \ell(w(i,\overline{\jmath}))-\ell(w)\,,\]
by the corresponding formula above.
\end{proof}

In relation to the criteria in Proposition \ref{critc}, we will use the same terminology as in Remark \ref{across} related to transposing two values across another one.

\subsection{Kashiwara-Nakashima columns}\label{sectkn} The {\em Kashiwara-Nakashima (KN) columns} \cite{kancgr} of \linebreak height $k$ index the vertices of the fundamental representation $V(\omega_k)$ of the symplectic algebra $\mathfrak{sp}_{2n}({\mathbb C})$, with root system of type $C_n$. 

\begin{definition} A column-strict filling $C=x_1\ldots x_k$ with entries in $[\overline{n}]$ is a KN column if there
is no pair $(z,\overline{z})$ of letters in $C$ such that: 
\[z = x_p\,,\;\;\;\;\;\overline{z} = x_q\,,\;\;\;\;\;q-p\le k - z\,.\]
\end{definition}

We will need a different definition of KN columns, which goes back to \cite{decsst}, and was proved to be equivalent to the one above in \cite{shesjt}.

\begin{definition}\label{defkn}
 Let $C$ be a column and $I=\{z_1 > \ldots > z_r\}$ the set of unbarred letters $z$ such that
the pair $(z,\overline{z})$ occurs in $C$. The column $C$ can be split when there exists a set
of $r$ unbarred letters $J = \{t_1 > \ldots > t_r\} \subset[{n}]$ such that:
\begin{itemize}
\item $t_1$ is the greatest letter in $[n]$ satisfying: $t_1 < z_1$, $t_1\not\in C$, and $\overline{t_1}\not\in C$,
\item for $i = 2, ..., r$, the letter $t_i$ is the greatest one in $[n]$ satisfying $t_i < \min(t_{i-1},z_i)$, $t_i\not\in C$, and $\overline{t_i}\not\in C$.
\end{itemize}
In this case we write:
\begin{itemize}
\item $rC$ for the column obtained by changing  $\overline{z_i}$ into $\overline{t_i}$ in $C$ for each letter $z_i\in I$, and by reordering if
necessary,
\item $lC$ for the column obtained by changing $z_i$ into $t_i$ in $C$ for each letter $z_i\in I$, and by reordering if
necessary.
\end{itemize}
The pair $(rC,lC)$ will be called a split column.
\end{definition}

\begin{example}\label{exkn}{\rm
The following is a KN column of height $5$ in type $C_n$ for $n\ge 5$, together with the corresponding split column:
\[C=\tableau{{4}\\{5}\\{\overline{5}}\\{\overline{4}}\\{\overline{3}}}\,,\;\;\;\;\;(rC,lC)=\tableau{{4}&{1}\\{5}&{2}\\{\overline{3}}&{\overline{5}}\\{\overline{2}}&{\overline{4}}\\{\overline{1}}&{\overline{3}}}\,.\]
We used the fact that $I=\{5>4\}$, so $J=\{2>1\}$. 
}
\end{example}

We will consider Definition \ref{defkn} as the definition of KN columns. In fact, for most of this paper we will work with split columns instead of KN columns, as we now explain.

Given our fixed dominant weight $\mu$, let $B_\mu$ denote the set of fillings of the shape $\mu$ with integers in $[\overline{n}]$ whose columns are KN columns. Like in type $A$, we can write
\begin{equation}\label{bmuc}B_\mu=\bigotimes_{i=\mu_1}^{1} B^{\mu_i',1}\,,\end{equation}
where $B^{k,1}$ is the traditional notation for the type $C_{n}$ {\em Kirillov-Reshetikhin crystal} indexed by a column of height $k$. Indeed, the KN columns of height $k$ are also an indexing set for $B^{k,1}$. As mentioned above, it will be more useful to realize the tensor factors $B^{k,1}$ in terms of split columns of height $k$; thus, an element of $B_\mu$ is a filling $D_r^{\mu_1}D_l^{\mu_1}\ldots D_r^1D_l^1$ of the shape $2\mu$, where $(D_r^j,D_l^j)=(rD^j,lD^j)$ is the splitting of a KN column $D^j$ of height $\mu_j'$. Using this realization and imposing the condition that the rows of a filling are weakly increasing (from right to left, in our display), we obtain the Kashiwara-Nakashima tableaux of shape $\mu$ with entries in $[\overline{n}]$; these index the vertices of the crystal $B(\mu)$ of highest weight $\mu$ inside $B_\mu$. 

\section{The type $C$ charge}\label{chc}

\subsection{The main construction}\label{constrc} The following is the main result in type $C$, and this section sketches its proof, which will be completed in Sections \ref{coll} and \ref{colr}. We use the setup in Section \ref{sectc}, as well as the map ``${\rm ord}$'' on fillings (cf. Section \ref{constra}), which is sorting each column increasingly.

\begin{theorem}\label{bijc} The following composite map is a bijection:
\begin{equation}\label{compc}\overline{\mathcal F}(\mu)\stackrel{f}{\longrightarrow} f(\overline{\mathcal F}(\mu))\stackrel{{\rm ord}}{\longrightarrow} B_\mu\,.\end{equation}
\end{theorem}

\begin{remarks}\label{remroots} (1) We can make essentially the same remarks as in type $A$, i.e., Remarks \ref{remfilla}. In other words, the above bijection will be realized as a crystal isomorphism in a future publication; on another hand, its restriction to those $(w,J)$ which yield paths in the Bruhat graph is a bijection to Kashiwara-Nakashima tableaux of shape $\mu$ with entries in $[\overline{n}]$. 

(2) There are two new features in the proof of Theorem \ref{bijc}, compared to its type $A$ counterpart, namely Theorem \ref{bija}. The obvious one is related to the ``doubled columns'' of the fillings in the image of the map $f$. A more subtle new feature is that the reflections in the segment $\Gamma_{ki}$ of $\Gamma_l(k)$ can change any of the first $i$ entries of a signed permutation (see the figure at the beginning of Section \ref{coll}).  By contrast, in type $A$, the reflections in the $i$th row of $\Gamma(k)$ in its display as a matrix (\ref{omegakchain}) only change the $i$th entry among the first $k$ entries of a permutation, cf. Algorithm \ref{algchain}. Note that the mentioned segment of $\Gamma(k)$ is the tail of $\Gamma_{ki}$ (upon identifying the type $A$ roots/reflections with the corresponding ones in type $C$). These two features add complexity to the proof of Theorem \ref{bijc}.
\end{remarks}

In type $C$, we will refer to a column as an increasing sequence of elements in $[\overline{n}]$ containing no pair $i\overline{\imath}$. Recall Conditions \ref{cond1} and \ref{cond2} on a pair of adjacent columns $C'C$ filled with elements in $[n]$. These conditions can be extended in a straightforward way to type $C$ columns, and they are still equivalent. We will refer to them as Conditions \ref{cond1}$'$ and \ref{cond2}$'$, respectively. 

\begin{remark}\label{remcondc} Let $\tau$ be a filling in $B_\mu$, represented with split columns (as a filling of the shape $2\mu$). Essentially the same algorithm as in Remark \ref{remcond} constructs the unique filling $\sigma$ satisfying the following conditions: (i) the first column of $\sigma$ (of height $\mu_1'$) is increasing, (ii) the adjacent columns of $\sigma$ satisfy Condition \ref{cond1}$'$, and (iii) ${\rm ord}(\sigma)=\tau$.
\end{remark} 

The proof of Theorem \ref{bijc} consists essentially of the following result.

\begin{proposition}\label{chainc}
The restriction of the filling map $f$ to $\overline{\mathcal F}(\mu)$ is injective. The image of this restriction consists of all fillings $C_r^{\mu_1}C_l^{\mu_1}\ldots C_r^1C_l^1$ of the shape $2\mu$ with integers in $[\overline{n}]$ satisfying the following conditions: {\rm (i)} $C_l^1$ is  increasing; {\rm (ii)} $({\rm ord}(C_r^j),{\rm ord}(C_l^j))=(rD^j,lD^j)$ for some KN column $D^j$, for all $j$; {\rm (iii)} any two adjacent columns satisfy Condition {\rm \ref{cond1}}$'$. 
\end{proposition}

The proof of Proposition \ref{chainc} is based on the following two results, which will be proved in Sections \ref{coll} and \ref{colr}, respectively. Recall the definition of the sequences of roots $\Gamma_l(k)$ and $\Gamma_r(k)$ from Section \ref{setc}. The reverse of a sequence $S$ is denoted by $\rev(S)$. 

\begin{proposition}\label{proplc} Consider a signed permutation $u$ in $B_n$, the column $C:=u[1,k]$, and another column $C'$ of height $k$. The pair of columns $C'C$ satisfies Condition {\rm \ref{cond1}$'$} if and only if there is a path $u=u_0,u_1,\ldots,u_p=v$ in the corresponding quantum Bruhat graph such that $v[1,k]=C'$ and the edge labels form a subsequence of $\rev(\Gamma_l(k))$. Moreover, the mentioned path is unique, and we have 
\begin{equation}\label{monotc}C(i)=u_0(i)\preceq u_1(i)\preceq \ldots\preceq u_p(i)=C'(i)\,, \;\;\;\;\;\mbox{for $i=1,\ldots,k$}\,.\end{equation} 
\end{proposition}

\begin{proposition}\label{proprc} Consider a signed permutation $u$ in $B_n$, the column $C:=u[1,k]$, and another column $C'$ of height $k$. The pair of columns $C'C$ satisfies Condition {\rm \ref{cond1}$'$} and $({\rm ord}(C'),{\rm ord}(C))=(rD,lD)$ for some KN column $D$ if and only if there is a path $u=u_0,u_1,\ldots,u_p=v$ in the corresponding quantum Bruhat graph such that $v[1,k]=C'$ and the edge labels form a subsequence of $\rev(\Gamma_r(k))$. 
Moreover, the mentioned path is unique and, for each $i=1,\ldots,k$, we have the following weakly increasing sequence in $[n]$ or $[\overline{n}]\setminus [n]$:
\begin{equation}\label{monotcr}C(i)=u_0(i)\le u_1(i)\le \ldots\le u_p(i)=C'(i)\,.\end{equation}
\end{proposition}

\begin{proof}[Proof of Proposition {\rm \ref{chainc}}] Consider a filling $\sigma=C_r^{\mu_1}C_l^{\mu_1}\ldots C_r^1C_l^1$ satisfying  conditions (i)-(iii), and let $u$ be the identity permutation. Apply Proposition \ref{proplc} to $u$ and the column $C_l^1$; then set $u$ to the output signed permutation $v$, and apply Proposition \ref{proprc} to $u$ and the column $C_r^1$. Continue in this way, by considering the columns $C_l^2,C_r^2,\ldots,C_l^{\mu_1},C_r^{\mu_1}$, while each time setting $u$ to the previously obtained $v$. This procedure constructs the unique folding pair $(w,T)$ mapped to $\sigma$ by the filling map $f$ (note that $T$ is constructed in reverse, and $w$ is the last $v$). Viceversa, a similar application of the two propositions shows that any filling in the image of $f$ satisfies conditions (i)-(iii).
\end{proof}

\begin{proof}[Proof of Theorem {\rm \ref{bijc}}] This is now immediate, based on Remark \ref{remcondc} and Proposition \ref{chainc}.
\end{proof}

\subsection{The construction of the type $C$ charge} The results in the previous section allow us to translate the computation of the statistic $\level(w,T)$ in the Ram-Yip formula (\ref{ryform}) into an algorithm which is similar to the one for the type $A$ charge. This translation is analogous to the corresponding one in type $A$, which is explained in detail in Section \ref{redcha}; therefore, below we concentrate on the new features in type $C$. 

Consider a filling $\tau$ in $B_\mu$, represented with split columns (as a filling of the shape $2\mu$); its columns are labeled from longest to shortest by $1,1',2,2',\ldots$. Consider the preimages of $\tau$ under the bijections in (\ref{compc}), as follows: 
\begin{equation}\label{mapc}(w,T)\mapsto \sigma=f(w,T) \mapsto \tau={\rm ord}(\sigma)\,.\end{equation}
Define the charge word $\cw(\tau)$ of $\tau$ by analogy with type $A$, as the biword containing a biletter $\binom{k}{j}$ for each entry $k$ in column $j$ of $\tau$; here $j$ and $k$ belong to the alphabets $\{1<1'<2<2'<\ldots\}$ and $[\overline{n}]$, respectively. We order the biletters as in the type $A$ case (in the decreasing order of the $k$'s, and for equal $k$'s, in the decreasing order of $j$'s), and define $\cw_2(\tau)$ in the same way (as the word formed by the lower letters $j$). 

The algorithm mentioned in Remark \ref{remcondc} for reconstructing $\sigma$ from $\tau$ can be rephrased in terms of $\cw_2(\tau)$, as explained below; we will refer to this rephrasing as the charge algorithm. We start by scanning $\cw_2(\tau)$ from right to left and by selecting the entries $1,1',2,2',\ldots,\mu_1,(\mu_1)'$ in this order, according to the same rule as in type $A$: always pick the first available entry to the left, but if the desired entry is not available then scan the word from its right end once again. As in the proof of Theorem \ref{levch}, we can see that the sequence of top letters paired with $1,1',2,2',\ldots,\mu_1,(\mu_1)'$ is the first row of the filling $\sigma$ (read from right to left). We then remove the selected entries from $\cw_2(\tau)$ and repeat the above procedure, which will now give the other rows of $\sigma$, from top to bottom. Note that, by Proposition \ref{proprc}, we always go left from $j$ to $j'$, but we can go right from $j'$ to $j+1$.

\begin{example}\label{exchc}{\rm Consider the following tensor product of KN columns
\[\tableau{{1}\\{3}\\{\overline{3}}}\otimes \tableau{{3}\\{\overline{4}}\\{\overline{3}}}\otimes \tableau{{\overline{5}}\\{\overline{3}}\\{\overline{2}}\\{\overline{1}}}\,.\]
This is represented with split columns as the following filling $\tau$ of the shape $2\mu=(6,6,6,2)$:
\[\tableau{{3'}&{3}&{2'}&{2}&{1'}&{1}\\ \\{1}&{1}&{3}&{2}&{\overline{5}}&{\overline{5}}\\{3}&{2}&{\overline{4}}&{\overline{4}}&{\overline{3}}&{\overline{3}}\\
{\overline{2}}&{\overline{3}}&{\overline{2}}&{\overline{3}}&{\overline{2}}&{\overline{2}}\\
&&&&{\overline{1}}&{\overline{1}}
}\,,\]
where the top row consists of the column labels. The corresponding filling $\sigma$ is
\[\tableau{{\overline{2}}&{\overline{3}}&{\overline{4}}&{\overline{4}}&{\overline{5}}&{\overline{5}}\\{1}&{1}&{\mathbf{\overline{2}}}&{\overline{3}}&{\overline{3}}&{\overline{3}}\\
{{3}}&{{2}}&{\mathbf{3}}&{{2}}&{\mathbf{\overline{2}}}&{\overline{2}}\\
&&&&{\overline{1}}&{\overline{1}}
}\,.\] 
Meanwhile, the corresponding charge word $\cw(\tau)$, with the order on the biletters indicated above, is
\[\left(\begin{array}{cccccccccccccccccccc}\overline{1}&\overline{1}&\overline{2}&\overline{2}&\overline{2}&\overline{2}&\overline{3}&\overline{3}&\overline{3}&\overline{3}&\overline{4}&\overline{4}&\overline{5}&\overline{5}&3&3&2&2&1&1\\
1'_4&1_4& 3'_1&2'_2&1'_3&1_3&  3_1&2_2&1'_2&1_2&  2'_1&2_1&  1'_1&1_1&  3'_3&2'_3&  3_3&2_3&  3'_2&3_2
\end{array}\right)   \,.\] 
The index attached to a lower letter is the number of the iteration in which the given letter is
selected by the charge algorithm.
}
\end{example}

Now let us refer to the computation of $\level(w,T)$, for the folding pair $(w,T)$ in (\ref{mapc}). By Proposition \ref{critc}, the set $T^-$ consists only of roots $(i,m)$ or $(i,\overline{\imath})$. Such a root $(i,\cdot)=\beta_s$ (where $s$ is the position of the root in the fixed $\mu$-chain $\Gamma$) must lie in $T^{j+1}_l$ for some $j$. This means that there is a descent in column $C^{j}_r$, labeled $j'$, and row $i$ of $\sigma=C^{\mu_1}_rC^{\mu_1}_l\ldots C^{1}_rC^{1}_l$ (descents are defined as usual, cf. Section \ref{redcha}); in other words, $C_{r}^{j}(i)>C_l^{j+1}(i)$. In fact, the above correspondence is a bijection between $T^-$ and the set ${\rm Des}(\sigma)$ of descents in $\sigma$, by Proposition \ref{proplc} and in particular (\ref{monotc}). Moreover, the affine level $l_s$ of the root $(i,\cdot)=\beta_s$ indicated above (see Section \ref{alcovewalks}) is half the arm length of the corresponding descent, cf. the proof of Theorem \ref{levch}. Finally, the descents in $\sigma$ correspond to the steps to the right in the charge algorithm applied to $\cw_2(\tau)$. This leads to the following definition of the type $C$ charge.

\begin{definition} Consider a word $w$ with letters in the alphabet $1,1',2,2',\ldots$, containing as many letters $j$ as $j'$, and at least as many letters $j$ as $j+1$. Apply the charge algorithm to $w$, and assume that a selected entry $j'$ is always to the left of the previously selected $j$. Let $\charge(w)$ be the sum of $k-j$ for each selected entry $j+1$ to the right of the previously selected $j'$, where the selected entries in the given iteration are $1,1',\ldots,k,k'$. 
\end{definition}

The above discussion leads to the following result.

\begin{theorem}\label{levchc} We have
\begin{equation}\label{chccomp}\level(w,T)=\frac{1}{2}\sum_{c\in{\rm Des}(\sigma)}{\rm arm}(c)=\charge(\cw_2(\tau))\,.\end{equation}
\end{theorem}
  
\begin{example}{\rm This is a continuation of Example \ref{exchc}. The entries in the descents of $\sigma$ are shown above in bold. Correspondingly, the charge algorithm applied to $\cw_2(\tau)$ makes one step to the right in the second iteration (from $2'$ to $3$), and two steps to the right in the third iteration (from $1'$ to $2$ and from $2'$ to $3$). Thus, $\charge(\cw_2(\tau))=1+(2+1)=4$. 
}
\end{example}

Let us replace, for simplicity, the notation $\charge(\cw_2(\tau))$ by $\charge(\tau)$. Recall from Section \ref{setcfill} the definition of the content of a filling with entries in $[\overline{n}]$. We obtain the following corollary of Theorem \ref{levchc} and the Ram-Yip formula (\ref{ryform}), based on Proposition \ref{weightmonc}.

\begin{corollary}
In type $C$, we have
\[P_{\mu}(X;q,0)=
\sum_{\tau\in B_\mu}q^{\charge(\tau)}\,x^{\content(\tau)}\,.\]
\end{corollary}

\section{The proof of Proposition \ref{proplc}}\label{coll}

Much of this proof relies on the structure of the sequence of roots $\rev(\Gamma_{ki})$, which determines the sequence of reflections in the corresponding quantum Bruhat path. It is useful to visualize such sequences as in the figure below. Here we display a signed permutation with its positions labeled by the elements of $[\overline{n}]$. The reflections in the mentioned sequence swap the entries in positions $i$ and $\overline{\imath}$ with entries in other positions in four stages I$-$IV; within a stage, the order of the reflections is indicated by the corresponding arrows.

\pagebreak

\begin{equation}\label{stages} \; \end{equation}

\vspace{-21cm} 
\[\!\!\!\!\!\!\!\!\!\!\!\!\mbox{\includegraphics[scale=0.83]{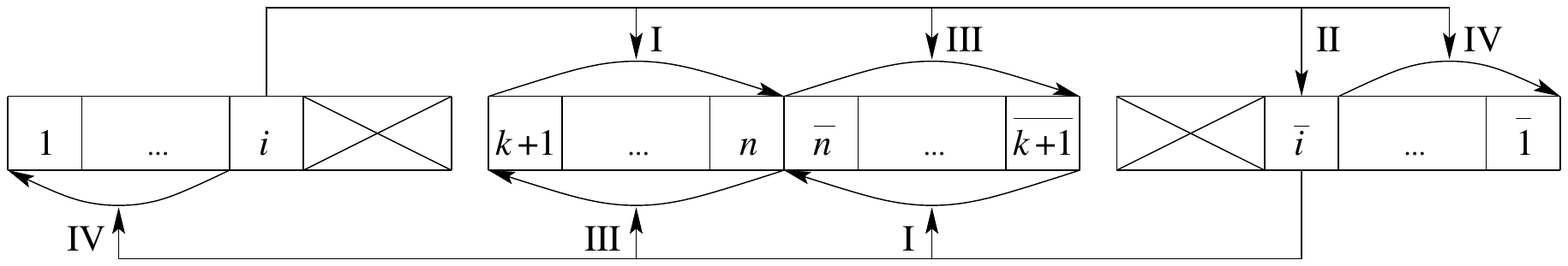}}\]

\subsection{The necessity of Condition \ref{cond1}$'$}\label{nec-cond1}

Consider a signed permutation $u$ in $B_n$ and a subsequence $T$ of $\rev(\Gamma_l(k))$ labeling a path $u=u_0,u_1,\ldots,u_p$ in the quantum Bruhat graph of type $C_n$, where $1\le k\le n$. Assume that (\ref{monotc}) holds. The factorization of $\rev(\Gamma_l(k))$ as $\rev(\Gamma_{kk})\ldots\rev(\Gamma_{k1})$ induces a factorization of $T$ as $T_k\ldots T_1$. Let $\pi_i:=uT_kT_{k-1}\ldots T_i$, for $i=k+1,k,\ldots,1$, cf. the notation in Sections \ref{secta} and \ref{sectc}; in particular, $\pi_{k+1}=u$. Define the corresponding columns $C_i:=\pi_i[1,k]$, and let $C:=C_{k+1}$, $C':=C_1$. 

We start with the following lemma.

\begin{lemma}\label{moveup} Fix $i$ and $l$ with $1\le i<l\le k$. Let $a$ be the entry in position $i$ of the signed permutation obtained at some point in the process of applying to $u$ the reflections in $T_k,$ $T_{k-1},\ldots,T_{l+1}$. Then either $a$ appears in $C_{l+1}[1,i]$ or $\overline{a}$ appears in $C_{l+1}[l+1,k]$.
\end{lemma}

\begin{proof} Assume that the letter $a$ gets moved from position $i$ by subsequent reflections, and that $(i,\overline{m})$ in $T_m$ is the first reflection moving it, where $l<m\le k$. Then either $\overline{a}$, which is moved to position $m$, is no longer moved by the next reflections in $T_m,T_{m-1},\ldots,T_{l+1}$, or a subsequent reflection in $T_m$ will move $a$ to a position $i'<i$. The reasoning continues in this way. 
\end{proof}

We now prove the necessity of Condition \ref{cond1}$'$ in Proposition \ref{proplc}.

\begin{proposition}\label{neccond1} The pair of columns $C'C$ satisfies Condition {\rm \ref{cond1}$'$}.
\end{proposition}

\begin{proof} Let us first assume that we have $C(i)=C'(l)=a$ for $1\le i<l\le k$. Let us examine what happens to the entry $a$ as we apply to $u$ the reflections in $T_k,T_{k-1},\ldots,T_{l+1}$. By Lemma \ref{moveup}, either $a$ appears in  $\pi_{l+1}[1,i]$ or $\overline{a}$ appears in $\pi_{l+1}[l+1,k]=C'[l+1,k]$. The second case would lead to both $a$ and $\overline{a}$ appearing in $C'$, so it is impossible.  In the first case, it is  impossible to move $a$ to position $l$ of  $\pi_l=\pi_{l+1}T_l$, as we should. We conclude that the first part of Condition \ref{cond1}$'$ holds.

Let us now assume that we have $C(i)\prec C'(l)=b\prec C'(i)$ for $1\le i<l\le k$. We cannot have $C_l(i)\preceq b\prec C'(i)$. Indeed, the entry in position $i$ of the signed permutation would then change from $C_l(i)$  to $C'(i)$ via reflections in $T_{l-1},T_{l-2},\ldots,T_i$; so one of these reflections would transpose entries across $b$, violating the quantum Bruhat graph criterion. In fact, we cannot have $C_{l+1}(i)\preceq b\prec C_l(i)$ either (if this were true, just examine the signs of the three values based on Proposition \ref{critc} (2)). So we must have $C(i)\prec b\prec C_{l+1}(i)$, by (\ref{monotc}). This means that, at some point in the process of applying to $u$ the reflections in $T_k,T_{k-1},\ldots,T_{l+1}$, we apply to the current signed permutation $w$ a reflection $(i,\overline{m})$ such that $w(i)=a< b<w(\overline{m})=c$ (here we used Lemma \ref{moveup} to rule out having the value $b$ in position $i$ during the mentioned process). Let  $(i_0,\overline{m}),\ldots,(i_p,\overline{m})$ be the segment of $T_m$ starting with $(i,\overline{m})$ and consisting of roots $(\cdot,\overline{m})$, where $i=i_0>i_1>\ldots>i_p\ge 1$. Let $a':=w(i_p)$. Note that $a'\le a<b<c$, and that all these values have the same sign, by Proposition \ref{critc} (2). We have $\pi_m(\overline{m})=a'$, and this entry is not changed by the reflections in $T_{m-1},\ldots,T_l$, which we examine next. Another important observation is that $\pi_m[i,\overline{m}]$ contains no entry (strictly) between $a'$ and $c$, by the quantum Bruhat graph criterion. Based on this observation, we can see that the reflections in $T_{m-1},\ldots,T_l$ can never bring a value between $a'$ and $c$ in positions $l,l+1,\ldots,m-1$ and $k+1,\ldots,n,\overline{n},\ldots,\overline{k+1}$; indeed, this could only happen if two values would be transposed across $a'$ in position $\overline{m}$, which would violate the quantum Bruhat graph criterion. But then $b$ cannot be moved to position $l$, which is a contradiction. We conclude that the second part of Condition \ref{cond1}$'$ also holds.
\end{proof}

\subsection{Constructing a segment of the quantum Bruhat path}\label{constr-segm} In this section we prove the following result, which underlies the construction of the path in the quantum Bruhat graph in Proposition \ref{proplc}. 

\begin{proposition}\label{baseconstr} Let $u,k,C,C'$ be as in Proposition {\rm \ref{proplc}}, and assume that the pair of columns $C'C$ satisfies Condition {\rm \ref{cond1}$'$}. Assume also that $C[i+1,k]=C'[i+1,k]$ for some $i$ with $1\le i\le k$. Then there is a unique path $u=u_0,u_1,\ldots,u_q=v$ in the corresponding quantum Bruhat graph such that $v(i)=C'(i)$ and the edge labels form a subsequence of $\rev(\Gamma_{ki})$. Moreover, we have
\begin{align}\label{monotc1}
&C(i)=u_0(i)\prec u_1(i)\prec\ldots\prec u_q(i)=C'(i)\,,\;\;\;\mbox{and}\\
&\mbox{if $u(l)\ne v(l)$, then  $C(l)=u(l)\prec v(l)\preceq C'(l)$ and}\label{monotc2}\\ 
&\;\;\;\;\mbox{$\sign(u(l))=\sign(v(l))$}, \;\;\;\mbox{for $l=1,\ldots,i-1$} \,.\nonumber\end{align}
\end{proposition}

We distinguish the following two cases, to which we refer freely below.

{\em Case {\rm 1.}} $C(i)\preceq C'(i)\prec \overline{C'(i)}$. This case has the four subcases shown in the figure below. Here the elements of $[\overline{n}]$ are placed on a circle as indicated, and the marked arc is oriented from $C(i)$ to $C'(i)$. 

\vspace{-17.4cm} 

\[\;\;\;\;\;\;\;\;\mbox{\includegraphics[scale=0.69]{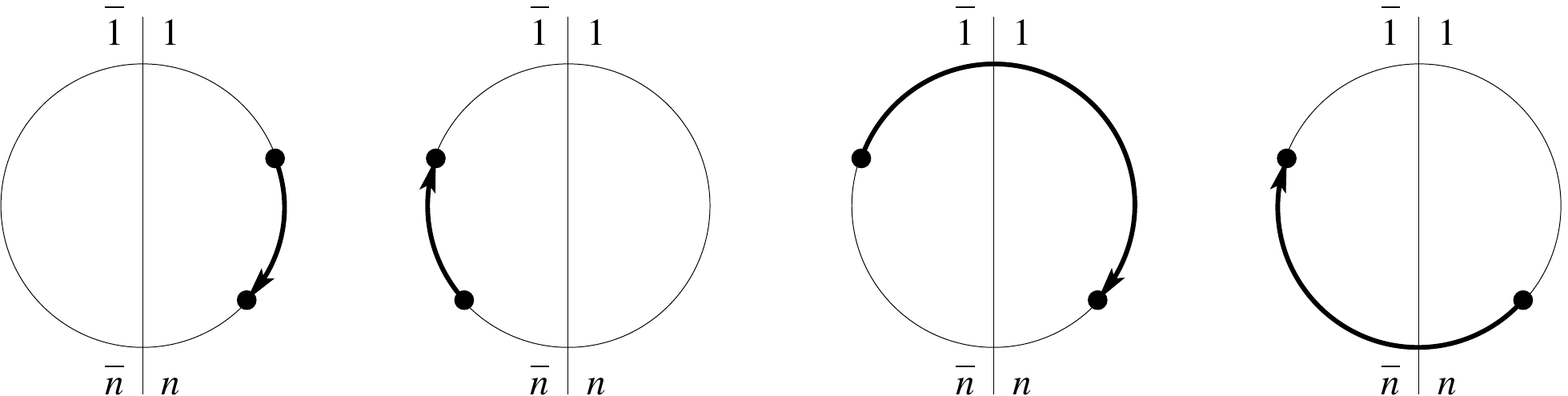}}\]

{\em Case {\rm 2.}} $C(i)\preceq \overline{C'(i)}\prec {C'(i)}$. This case has the four subcases shown in the figure below, where the same notation as above is used.

\pagebreak 
$\;$

\vspace{-18cm}

\[\;\;\;\;\;\;\;\;\mbox{\includegraphics[scale=0.69]{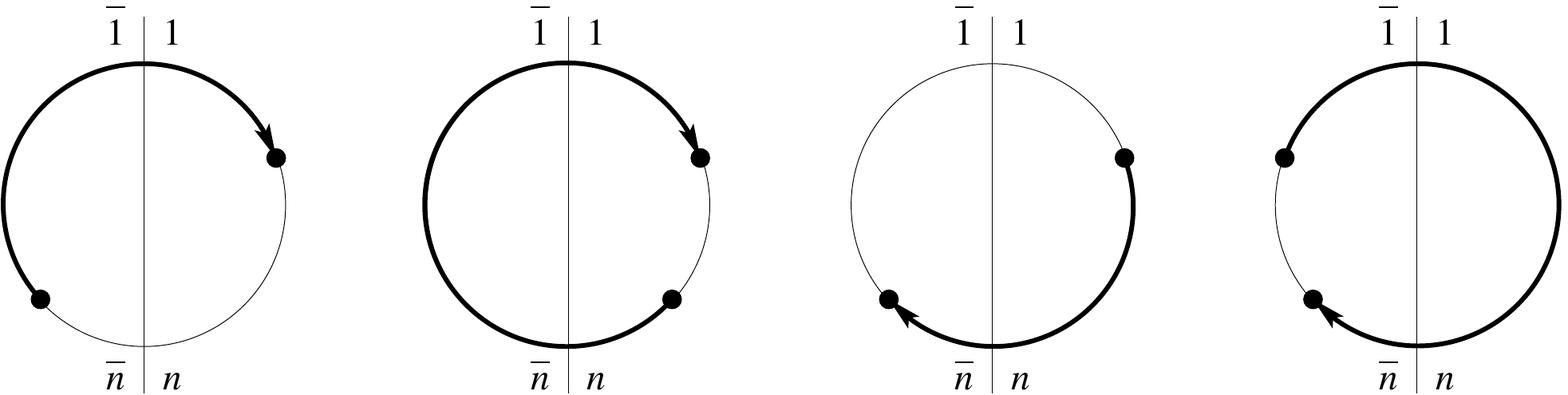}}\]

The next lemma will be used in the proof of Proposition \ref{baseconstr}. We need the following notation:
\[M(u,i,C'):=\max\,\{u(i)\}\cup\{u(l)\::\:\#C'+1\le l\le n,\:u(i)\prec u(l)\preceq C'(i)\}\,,\]
where the maximum is taken with respect to the circular order $\prec_{u(i)}$ on $[\overline{n}]$. 

\begin{lemma}\label{casetwo} Under the hypotheses of Proposition {\rm \ref{baseconstr}}, in Case {\rm 2} we have 
\[\overline{C'(i)}\preceq M(u,i,C')\preceq C'(i)\,.\]
\end{lemma}

\begin{proof} Without loss of generality, assume one of the first two subcases of Case 2, where $C(i)\ne\overline{C'(i)}$. Let $a:=C'(i)\in[n]$ and $A:=\{\overline{a},\overline{a-1},\ldots,\overline{1},1,\ldots,a\}$. We need to show that $u[k+1,n]$ contains an element in $A$, so assume the contrary. Note that $u[i+1,k]=C'[i+1,k]$ contains no element in $A$, by Condition \ref{cond1}$'$. We conclude that $u[1,i-1]$ contains an element from each pair $\{x,\overline{x}\}$ of elements in $A$. If $u(i')\in A$ for $i'<i$, we say that $u(i')$ is matched with $C'(i')$. Since $C'(i)=a$, the only possible matches for $u[1,i-1]\cap A$ are elements in $A\setminus\{a,\overline{a}\}$, by Condition \ref{cond1}$'$. But these are too few to match $a$ elements, which is a contradiction.
\end{proof}

Assume for the moment that a path with the property stated in Proposition \ref{baseconstr} exists. Let $T$ be the sequence of edge labels  for this path. Note that the sequence of roots $\rev(\Gamma_{ki})$ has the sequence $(i,k+1),\ldots,(i,n)$ as its head (the latter happens to be the $i$th row of the type $A$ chain $\Gamma(k)$ in its display as a matrix (\ref{omegakchain}), cf. Remark \ref{remroots}). The factorization of $\rev(\Gamma_{ki})$ into the mentioned head and the corresponding tail induces a factorization of $T$ denoted $T_AT_C$. Let $u_A:=uT_A$ and
\[u_A':=\casetwo{u_A(i,\overline{\imath})}{(i,\overline{\imath})\in T_C}{u_A}\]

The following lemma gives necessary conditions for the construction in Proposition \ref{baseconstr}.

\begin{lemma}\label{lemrefla}
We have $u_A(i)=M(u,i,C')$. Moreover, the root $(i,\overline{\imath})$ belongs to $T_C$ (being its first root) precisely when ${\rm sign}(u_A(i))\ne{\rm sign}(C'(i))$, and either $u_A'(i)\le C'(i)\le n$ or $\overline{n}\le u_A'(i)\le C'(i)$. 
\end{lemma}

\begin{proof}
To prove the second statement of the lemma, note that, by Proposition \ref{critc} (2), the reflections in $T_C\setminus(i,\overline{\imath})$ can only increase the value in position $i$, while preserving the sign.  

For the first statement, assume the contrary, which means that
\begin{equation}\label{notgetmax}M(u,i,C')\in u_A[k+1,n]\,,\;\;\;\;\;\mbox{and}\;\;\;\;\;u_A(i)\prec M(u,i,C')\preceq C'(i)\,.\end{equation}
We can quickly rule out $M(u,i,C')=C'(i)$. We now distinguish two cases, depending on the signs of $M(u,i,C')$ and $C'(i)$ coinciding or not. In the first case, the entry $M(u,i,C')$ in $u_A$ is not moved upon applying the reflections in $T_C$ (by the remark at the beginning of the proof). By (\ref{notgetmax}), one of the reflections in $T_C$ transposes values across $M(u,i,C')$, so the quantum Bruhat graph criterion is violated. In the second case, the second statement of the lemma and (\ref{notgetmax}) imply that ${\rm sign}(u_A(i))={\rm sign}(M(u,i,C'))$, that $T_C$ starts with $(i,\overline{\imath})$, and that this reflection transposes values across $M(u,i,C')$. So we end up with a contradiction once again.
\end{proof}

We now describe the algorithm that constructs the path in Proposition \ref{baseconstr}. The notation is the one introduced above. The algorithm inputs the signed permutation $u$, the target column $C'$, and the position $i$; it outputs the lists of reflections $T_A,T_C$ determining the path, and the end permutation $v$. The algorithm invokes twice the greedy procedure {\it path-A} (see Algorithm \ref{algchain}), which is used here in the type $C$ context, namely the set of entries $[n]$ is replaced by $[\overline{n}]$; the two calls of {\it path-A} refer to stages I and III$-$IV in (\ref{stages}), respectively. 

\begin{algorithm}\label{algchainc}\hfill\\
procedure path-C$(u,i,C')$;\\
let $M:=M(u,i,C')$, $k:=\#C'$, $L:=(k+1,\ldots,n)$;\\
$(T_A,u_A)$:=path-A$(u,i,M,L)$;\\
if ${\rm sign}(u_A(i))\ne{\rm sign}(C'(i))$ then \\
\indent let $T_C:=(i,\overline{\imath})$, $u_A':=u_A(i,\overline{\imath})$;\\
else let $T_C:=\emptyset$, $u_A':=u_A$;\\
end if;\\
let $L:=(\overline{n},\ldots,\overline{k+1},\overline{i-1},\ldots,\overline{1})$;\\
$(S,v)$:=path-A$(u_A',i,C'(i),L)$;\\
let $T_C:=T_C,S$;\\
return $((T_A,T_C),v)$;\\
end.
\end{algorithm}

\begin{remarks}\label{rem-chain} (1) The algorithm terminates (correctly). Indeed, by Condition \ref{cond1}$'$, the entry $C'(i)$ appears in $u$ in one of the positions $k+1,\ldots,n,\overline{n},\ldots,\overline{k+1}$ or $\overline{\imath},\ldots,\overline{1}$. 

(2) Lemma \ref{casetwo} guarantees that $u_A'(i)\le C'(i)\le n$ or $\overline{n}\le u_A'(i)\le C'(i)$. So all the edges of the constructed path corresponding to stages III and IV in (\ref{stages}) are up steps in Bruhat order. Moreover, it follows that the monotonicity condition (\ref{monotc1}) is satisfied. 

(3) The necessary conditions in Lemma \ref{lemrefla} are incorporated into the algorithm. The uniqueness of the path then follows by the arguments used in type $A$ (see the proof of Lemma \ref{chconst} (2)), which here show that only the greedy construction could work in stages I and III$-$IV. 

(4) The construction in the algorithm guarantees that an edge labeled $(i,\overline{\imath})$ is an edge of the quantum Bruhat graph. Furthermore, Condition \ref{cond1}$'$ and the greedy approach guarantee that all the other edges satisfy the conditions in the corresponding quantum Bruhat graph criteria which refer to the transposition of the value in position $i$ across the positions $i+1,\ldots,n,\overline{n},\ldots,\overline{k+1}$ and $\overline{\imath},\ldots,\overline{2}$ (see the proof of Lemma \ref{chconst} (2)).
\end{remarks}

\begin{proof}[Proof of Proposition {\rm \ref{baseconstr}}] Assume that $u(i)\ne C'(i)$ and that (\ref{monotc2}) holds. Then, given Remarks \ref{rem-chain}, the only fact left to prove is that the reflections pertaining to stage IV satisfy the condition in the corresponding quantum Bruhat graph criterion which refers to the transposition of the value in position $i$ across the positions $\overline{k},\ldots,\overline{i+1}$ (or, equivalently, to the transposition of the value in position $\overline{\imath}$ across the positions $i+1,\ldots,k$). Let us assume that at some point in stage IV we apply to the current signed permutation $w$ a reflection $(l,\overline{\imath})$ with $l<i$, such that 
\[w(l)<w(j)=C'(j)<\overline{w(i)}\,,\;\;\;\;\mbox{for some $j\in\{i+1,\ldots,k\}$}\,.\]
But, by (\ref{monotc2}), we have
\[C(l)=u(l)=w(l)\prec \overline{w(i)}\preceq C'(l)\,.\]
It follows that Condition \ref{cond1}$'$ for $C'C$ is violated. This completes the proof of the fact that Algorithm \ref{algchainc} constructs a path in the quantum Bruhat graph. 

We conclude the proof by addressing (\ref{monotc2}). Assume that this fails for some $l$, and let $l_1<i$ be the largest such $l$. This means that, if $w$ is the signed permutation to which the reflection $(l_1,\overline{\imath})$ is applied in stage IV, then
\[a:=w(i)<b:=\overline{C'(l_1)}<c_1:=\overline{w(l_1)}\,.\]
Let $\widetilde{C}:=w[1,k]$, and note that, by (\ref{monotc2}), the pair of columns $C'\widetilde{C}$ satisfies Condition \ref{cond1}$'$. Assume that the complete sequence of reflections applied to $w$ in stage IV is $(l_1,\overline{\imath}),\ldots,(l_p,\overline{\imath})$, where $l_1>l_2>\ldots>l_p$. Let $c_i:=\overline{w(l_i)}$, where $a<b<c_1<c_2<\ldots<c_p$ all have the same sign. Without loss of generality, assume that this sign is positive.

Given any $x\in\{b,b+1,\ldots,c_k\}$, either $x$ or $\overline{x}$ is in $\widetilde{C}$, say in position $j$, by Algorithm \ref{algchainc}. We claim that the possible values for $C'(j)$ are $\{\pm b,\pm(b+1),\ldots,\pm(c_k-1)\}$. Overall, we then have too few choices for these values, which is a contradiction. The claim is clearly true if $j>i$, because then $x\ne c_k$ and $\widetilde{C}(j)=C'(j)$. For $j<i$, we have the following two cases.

{\em Case $1$}: $\widetilde{C}(j)=x$. Since $x\ne c_k$ and $C'(i)=c_k$, we have $C'(j)\in\{x,x+1,\ldots,c_k-1\}$, based on Condition \ref{cond1}$'$.

{\em Case $2$}: $\widetilde{C}(j)=\overline{x}$. In this case, by Algorithm \ref{algchainc}, we have $j\le l_1$, so we can assume $j<l_1$. Since $C'(l_1)=\overline{b}$, we have $C'(j)\in\{\overline{x},\overline{x-1},\ldots,\overline{b+1}\}$, based on Condition \ref{cond1}$'$. In addition, if $x=c_k$, we have $C'(j)\ne\overline{c_k}$ because $C'(i)=c_k$. 

Note that $x=b$ cannot be in Case 2, because this would violate the first part of Condition \ref{cond1}$'$, given that $C'(l_1)=\overline{b}$. The claim related to the values of $C'(j)$ is now proved, which concludes the entire proof. 
\end{proof}

\subsection{Completing the proof of Proposition \ref{proplc}} What remains to be done is to construct the full chain in the statement, by iterating Algorithm \ref{algchainc}. 

\begin{proof}[Proof of Proposition {\rm \ref{proplc}}] Recall that the necessity of Condition \ref{cond1}$'$ was proved in Section \ref{nec-cond1}. To prove that Condition \ref{cond1}$'$ implies the existence of the chain in Proposition {\rm \ref{proplc}}, we iterate the construction in Proposition \ref{baseconstr}, based on Algorithm \ref{algchainc}, for $i=k,k-1,\ldots,1$. Note that, by (\ref{monotc2}) corresponding to iteration $i+1$, the pair of columns $C'C_{i+1}$ satisfies Condition \ref{cond1}$'$, so the hypothesis for iteration $i$ is satisfied (the notation is that at the beginning of Section \ref{nec-cond1}). Finally, note that (\ref{monotc}) follows from (\ref{monotc1}) and (\ref{monotc2}). 
\end{proof}

\section{The proof of Proposition \ref{proprc}}\label{colr}

Proposition \ref{proprc} refers to the subchain $\Gamma_r(k)$ of our chosen $\omega_k$-chain $\Gamma(k)=\Gamma_r(k)\Gamma_l(k)$. Note that $\Gamma_r(k)$ is a subchain of $\Gamma_l(k)$; more precisely, the action of the reflections in ${\rm rev}(\Gamma_r(k))$ correspond to stage IV in (\ref{stages}). The proof of Proposition \ref{proprc} is based on three intermediate results: Lemma \ref{lem1} does a reduction to Proposition \ref{proplc} and the related results, Lemma \ref{lem2} does a further reduction from the case of arbitrary columns to the one of sorted columns, and Proposition \ref{lem4} establishes the relationship with the condition defining KN columns, see Definition \ref{defkn}.

We start by fixing a signed permutation $u$ and a column $C'$ of height $k$. As usual, let $C:=u[1,k]$, and define
\[{\rm int}(C,C'):=\left(\bigcup_{i=1}^k\{j\in[\overline{n}]\::\:C(i)<j<C'(i)\}\right)\setminus \{\pm C(i)\::\:i=1,\ldots,k\}\,.\]
We need the following three conditions on the pair $C'C$.

\begin{condition}\label{condr1} We have
\[\{|C(i)|\::\:i=1,\ldots,k\}=\{|C'(i)|\::\:i=1,\ldots,k\}\,.\]
\end{condition}
\begin{condition}\label{condr2} For $i=1,\ldots,k$, we have
\[C(i)\le C'(i)\le n \;\;\;\mbox{or}\;\;\; \overline{n}\le C(i)\le C'(i)\,.\]
\end{condition}
\begin{condition}\label{condr3} We have ${\rm int}(C,C')=\emptyset$.
\end{condition}

\subsection{Reduction to Proposition \ref{proplc} and the related results}

We use the notation introduced above, and prove the following lemma.

\begin{lemma}\label{lem1} The pair $C'C$ satisfies Conditions {\rm \ref{cond1}$'$}, {\rm \ref{condr1}}, {\rm \ref{condr2}}, and {\rm \ref{condr3}} if and only if there is a path $u=u_0,u_1,\ldots,u_p=v$ in the corresponding quantum Bruhat graph such that $v[1,k]=C'$ and the edge labels form a subsequence of $\rev(\Gamma_r(k))$. Moreover, the mentioned path is unique, and {\rm (\ref{monotcr})} holds.
\end{lemma}

\begin{proof} The necessity of Condition {\rm \ref{cond1}$'$} is the content of Proposition \ref{neccond1}. The necessity of Condition {\rm \ref{condr1}}
is clear from the type of the roots in $\Gamma_r(k)$, while that of Conditions {\rm \ref{condr2}} and {\rm \ref{condr3}} from the corresponding quantum Bruhat graph criterion, namely Proposition {\rm \ref{critc} (2)}. Note that ${\rm int}(C,C')$ consists of the entries in positions $k+1,\ldots,n,\overline{n},\ldots,\overline{k+1}$ over which two values are transposed at some point in the given path (the entries in the mentioned positions are never moved). Viceversa, if all the mentioned conditions are satisfied, then we obtain the desired path by iterating Algorithm \ref{algchainc} (cf. Proposition \ref{proplc} and its proof). Indeed, Condition \ref{condr1} guarantees that, at any given point, the target entry $C'(i)$ cannot be in positions $k+1,\ldots,n,\overline{n},\ldots,\overline{k+1}$ of the current permutation; then Conditions \ref{condr2} and \ref{condr3} ensure that the algorithm never selects reflections of the form $(i,\overline{\imath})$, $(i,m)$, or $(i,\overline{m})$, where $m>k$, and therefore all the selected  reflections are in $\Gamma_r(k)$. 
\end{proof}

\subsection{Reduction to sorted columns}

Consider the ordered columns $D':={\rm ord}(C')$ and $D:={\rm ord}(C)$, where the notation is the same as above.

\begin{lemma}\label{lem2} Assume that the pair $C'C$ satisfies Conditions {\rm \ref{cond1}$'$} and {\rm \ref{condr1}}. 

{\rm (1)} Assuming that $C'C$ satisfies Condition {\rm \ref{condr2}} as well, the set ${\rm int}(C,C')$ contains an element $a$ if and only if it contains $\overline{a}$. 

{\rm (2)} The pair $C'C$ satisfies Conditions {\rm \ref{condr2}} and {\rm \ref{condr3}} if and only if the pair $D'D$ satisfies the same conditions. 
\end{lemma}

\begin{proof} For the first part, assume that we iteratively apply a modification of Algorithm \ref{algchainc} in which the positions $k+1,\ldots,n,\overline{n},\ldots,\overline{k+1}$ are ignored (so they never appear in a list $L$), cf. the proof of Lemma \ref{lem1}. This is clearly possible, by the above remarks. Moreover, it is still true that ${\rm int}(C,C')$ consists of the entries in positions $k+1,\ldots,n,\overline{n},\ldots,\overline{k+1}$ over which two values are transposed at some point in the above procedure. But these entries always occur in pairs $a,\overline{a}$.

For the second part, arrange all the entries in $C$ and $C'$ in increasing order, and match each entry $C(i)$, which we color red, with $C'(i)$, which we color blue. By Condition {\rm \ref{cond1}$'$}, this matching is obtained by scanning the red elements in the order $C(1),\ldots,C(k)$, and by matching each one with the leftmost unmatched blue entry to its right. We claim that the same algorithm can be applied for any order of the red elements, leading to other matchings; moreover, the sign of two matched entries is always the same. Indeed, let us swap the order of two red elements, say $a,b$, which were initially matched with $a'\ge a$ and $b'\ge b$, respectively. If the intervals $[a,a']$ and $[b,b']$ have empty intersection (in particular if $a$ and $b$ have different signs), the matching is not changed, whereas if they do, then $b$ becomes matched with $a'$ and $a$ with $b'$; the latter case happens precisely when $a<b\le a'<b'$ or $b<a\le a'<b'$. Note that if the red elements (in $C$) are scanned in increasing order, then the matched entries are in increasing order too. It follows that the pair $D'D$ satisfies Condition {\rm \ref{condr2}}. Condition {\rm \ref{condr3}} follows from the above local rule for changing a matching. Finally, the opposite implication is clear by the same reasoning.
\end{proof}

\subsection{The relationship with KN columns and the conclusion of the proof} We start with a simple construction. Consider two increasing sequences of integers (sorted columns) $A=\{a_1< a_2< \ldots<a_k\}$ and $B=\{b_1< b_2<\ldots<b_l\}$. Define the column ${\rm maxcol}(A,B)$ as the sorted column $C=\{c_1< c_2< \ldots<c_k\}$ given by
\[{\rm maxcol}(A,B):=\max\;\{C\::\:C\le A\,, C\cap B=\emptyset\}\,;\]
here the column comparison means entrywise comparison in ${\mathbb Z}$. It is easy to see that the entries $c_i$ are constructed recursively  by
\begin{equation}\label{match1}c_k:=\max\;\{c\in{\mathbb Z}\::\:c\le a_k\}\setminus B\,,\;\;\;\;c_i:=\max\;\{c\in{\mathbb Z}\::\:c\le a_i,\,c<c_{i+1}\}\setminus B\,,\end{equation}
for $i=k-1,\ldots,1$. 

\begin{lemma}\label{lem3} {\rm (1)} For all $i=1,\ldots,k$, we have
\[[c_i,a_i]\cap{\mathbb Z}\subseteq B\cup\{c_i,c_{i+1},\ldots,c_k\}\,.\]

{\rm (2)} We have $A\setminus B\subseteq {\rm maxcol}(A,B)$. Moreover, we have
\begin{equation}\label{eqlem3} {\rm maxcol}(A,B)=(A\setminus B)\sqcup {\rm maxcol}(A\cap B,A\cup B)\,,\end{equation}
where $\sqcup$ denotes disjoint union.
\end{lemma}

Before presenting the proof, we consider an example.

\begin{example}\label{exdoub} {\rm Let $A=\{3,4,5\}$ and $B=\{4,5\}$. The construction of ${\rm maxcol}(A,B)=\{1,2,3\}$ based on (\ref{match1}) can be represented graphically as in the left
figure below; the $a_i$ and $b_i$ are displayed as dots in the left and right columns, respectively, while the arrows represent the matches of $a_i$ with $c_i$. The similar matching corresponding to the right-hand side of (\ref{eqlem3}) is shown in the figure on the right.

\pagebreak

$\;$

\vspace{-17.7cm}


\[\;\;\;\;\;\;\;\;\;\;\;\;\;\;\;\;\;\;\;\;\;\;\;\;\;\;\;\;\;\;\;\;\;\;\;\;\;\mbox{\includegraphics[scale=0.7]{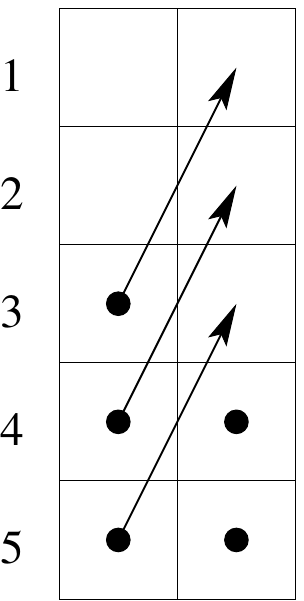}}\!\!\!\!\!\!\!\!\!\!\!\!\!\!\!\!\!\!\!\!\!\!\!\!\!\!\!\!\!\!\!\!\!\!\!\!\!\!\!\!\!\!\!\!\!\!\!\!\!\!\!\!\!\!\!\!\!\!\!\!\!\!\!\!\!\!\!\!\!\!\!\!\!\!\!\!\!\!\!\!\!\!\!\!\!\!\!\!\!\!\!\!\!\!\!\!\!\!\!\!\!\!\!\!\!\!\!\!\!\!\!\!\!\!\!\!\!\!\!\!\!\!\!\!\!\!\!\!\!\!\!\!\!\!\!\!\!\!\!\!\!\!\!\!\!\!\!\!\!\!\!\mbox{\includegraphics[scale=0.7]{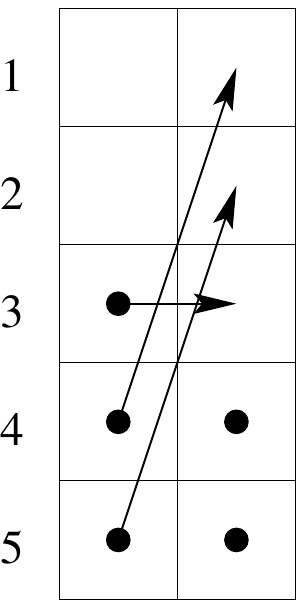}}\]
}
\end{example}

\begin{proof}[Proof of Lemma {\rm \ref{lem3}}] For the first part, assume that $x$ is an integer in $[c_i,a_i]$ which does not belong to $B\cup{\rm maxcol}(A,B)$. Then consider the largest $c_j$ in $[c_i,x]$, where $j\ge i$. Clearly, $c_j$ does not satisfy the maximality condition (\ref{match1}). 

The inclusion $A\setminus B\subseteq {\rm maxcol}(A,B)$ immediately follows from the first part. In order to prove (\ref{eqlem3}), we will show that, for any $x=a_i\in A\setminus B$, we have
\begin{equation}\label{remx}{\rm maxcol}(A,B)=\{x\}\sqcup {\rm maxcol}(A\setminus\{x\},B\sqcup\{x\})\,,\end{equation}
and then we apply this iteratively. As $x\in {\rm maxcol}(A,B)$, we have $c_i<c_{i+1}<\ldots<c_j=x$ for some $j\ge i$. Then the matching $(a_p,c_p)$, for $p=1,\ldots,k$, which corresponds to the left-hand side of (\ref{remx}), can be partially modified as follows: 
\[(x=a_i,c_j=x)\,,\;(a_j,c_{j-1})\,, \,\ldots\,(a_{i+2},c_{i+1})\,,\;(a_{i+1},c_i)\,.\]
But the latter matching corresponds to the right-hand side of (\ref{remx}) by (\ref{match1}). See Example \ref{exdoub} for an illustration of this procedure.
\end{proof}

Now consider a pair of sorted columns $D'D$ of the same height, with entries in $[\overline{n}]$. Split $D'$ and $D$ into their positive parts $D'_+$, $D_+$ (i.e., consisting of positive entries) and negative parts $D'_-$, $D_-$ (i.e., consisting of negative entries). Given a column $C$, we denote by $|C|$ the column obtained by taking the absolute value of all the entries in $C$ (and possibly reordering the new entries). 


\begin{proposition}\label{lem4} Given a pair of sorted columns $D'D$ of the same height, the following are equivalent:
\begin{enumerate}
\item the pair $D'D$ satisfies Conditions {\rm \ref{condr1}}, {\rm \ref{condr2}}, and {\rm \ref{condr3}};
\item we have 
\[|D'_-|={\rm maxcol}(|D_-|,D'_+)\,,\;\;\;\mbox{and}\;\;\; D_+=(D'_+\cup |D'_-|)\setminus |D_-|\,;\]
\item $D'$ and $D$ are respectively the right and left columns which represent the splitting of the KN column with positive part $D'_+$ and negative part $D_-$.
\end{enumerate}
\end{proposition}

Before presenting the proof, we consider an example. 

\begin{example}\label{exknmax}{\rm Consider the KN column $C$ in Example \ref{exkn}, and let $D':=rC$, $D:=lC$. Note that $|D_-|$ and $D'_+$ are precisely the columns $A$ and $B$ in Example \ref{exdoub}, respectively. Note that the construction of  $|D'_-|={\rm maxcol}(|D_-|,D'_+)$ given by the right-hand side of (\ref{eqlem3}), when viewed as the corresponding matching in Example \ref{exdoub}, is precisely the construction of the split column $(D',D)$ in Definition \ref{defkn}.
}
\end{example}

\begin{proof}[Proof of Proposition {\rm \ref{lem4}}] Note first the the equivalence of (2) and (3) is clear by (\ref{eqlem3}), because the construction of a split column in Definition \ref{defkn} is given by the right-hand side of (\ref{eqlem3}), see Examples \ref{exdoub} and \ref{exknmax}.

Under the assumptions of (1), let $|D_-|=\{d_1<d_2<\ldots<d_l\}$ and $|D'_-|=\{d_1'<d_2'<\ldots<d_l'\}$. To prove $(1)\Rightarrow (2)$, assume that the construction of $D'_-$ in (2) fails, so 
\[d_i'<\max \;\{d\in [n]\::\:d\le d_i,\,d<d_{i+1}'\}\setminus D_+'=:x\;\;\;\;\mbox{for some $i$}\,.\]
We cannot have $x\in |D_-|$, because then the construction of $D_+$ in (2) would produce a longer column than $D_+'$. But then $\overline{x}\in{\rm int}(D,D')$, so Condition \ref{condr3} fails.

To prove $(2)\Rightarrow (1)$, note first that, by Lemma \ref{lem3} (2), the constructions in (2) ensure that $\#D_-=\#D'_-$, $\#D_+=\#D'_+$, and $D_+\sqcup D_-=D_+'\sqcup D_-'$. It follows that Conditions (\ref{condr1}) and (\ref{condr2}) are satisfied. Using the same notation as above, by Lemma \ref{lem3} (1) we have
\[[d_i',d_i]\cap[n]\,\subseteq \,\{d_i',\ldots,d_l'\}\cup D'_+\;\;\;\;\mbox{for $i=1,\ldots,l$}\,,\]
which implies ${\rm int}(D,D')\setminus[n]=\emptyset$. Then ${\rm int}(D,D')=\emptyset$ by Lemma \ref{lem2} (1), so Condition \ref{condr3} is satisfied. 
\end{proof}

We now conclude the proof of Proposition \ref{proprc}.

\begin{proof}[Proof of Proposition {\rm \ref{proprc}}] Immediate based on Lemma \ref{lem1}, Lemma \ref{lem2} (2), and Proposition \ref{lem4}.
\end{proof}

\section{The charge in types $B$ and $D$}\label{typebd} We believe that the above results can be extended to types $B$ and $D$. More precisely, we claim that there is a bijection similar to the one in Theorem \ref{bijc} between the corresponding set of folding pairs $\overline{\mathcal F}(\mu)$ and the corresponding tensor product of KR-crystals $B_\mu$. In addition, we claim that, like in Theorem \ref{levchc}, this bijection leads to a translation of the statistic ``level'' in the Ram-Yip formula (\ref{ryform}) into a charge statistic, which expresses the corresponding energy function. These claims are supported by \cite{ST:2011}[Corollary 9.5] expressing a type $D$ Macdonald polynomial at $t=0$ in terms of the corresponding energy function. 

The above problems in types $B$ and $D$ display additional complexity, due to some new aspects, that we now describe. We start by referring to type $B_n$, as type $D_n$ has all the complexity of type $D_n$ plus additional one. We label the positive roots in type $B_n$ and the corresponding reflections as in type $C_n$; in particular, the root $\varepsilon_i$ is labeled $(i,\overline{\imath})$. We fix $k$ with $1\le k\le n-1$, and note that there is an $\omega_k$-chain $\Gamma(k)$ very similar to the one in type $C$, see Section \ref{setc} and \cite{lenhhl}[Section 5]; namely  $\Gamma(k)=\widehat{\Gamma}_r(k){\Gamma}_l(k)$, where $\Gamma_l(k)$ is as in type $C_n$, and $\widehat{\Gamma}_r(k):=\widehat{\Gamma}_1\ldots\widehat{\Gamma}_k$, where $\widehat{\Gamma}_i$ is obtained from $\Gamma_i$ in type $C_n$ by appending to it the root $(i,\overline{\imath})$. For the corresponding KN columns, indexing the vertices of the crystals $B(\omega_i)$ corresponding to the fundamental representations $V(\omega_i)$, we refer to \cite{kancgr}. 

The first new aspect in type $B$ is related to the splitting of the KR crystal $B^{k,1}$ upon removing the 0-arrows as the following direct sum of (classical) crystals: $B(\omega_{k})\oplus B(\omega_{k-2})\oplus\ldots$. In terms of the folding pairs $(w,T)$ in $\overline{\mathcal F}(\mu)$, the mentioned phenomenon manifests itself in the existence of quantum edges in the paths corresponding to $(w,T)$, which was not the case in type $C$. We illustrate this based on the following example.

\begin{example}{\rm Let $\mu=\omega_4$ in $B_5$. Consider $(w,T)=(w,T_rT_l)$ with $w=3\overline{2}145$, which is uniquely determined by the path in the quantum Bruhat graph $\pi(w,T_r)$ below (read from right to left):
\[  w=\begin{array}{l}\tableau{{\mathbf 3}\\{\mathbf {\overline{2}}}\\{1}\\{4}}\\ \\ \tableau{{5}} \end{array}\!>\!
\begin{array}{l}\tableau{{2}\\{\mathbf {\overline{3}}}\\{1}\\{\mathbf 4}}\\ \\ \tableau{{5}} \end{array}\!>\!
\begin{array}{l}\tableau{{2}\\{\overline{4}}\\{\mathbf 1}\\{\mathbf 3}}\\ \\ \tableau{{5}} \end{array}\!<\!
\begin{array}{l}\tableau{{2}\\{\overline{4}}\\{\overline{3}}\\{\overline{1}}}\\ \\ \tableau{{5}} \end{array}.
\]
Here $T_r$ is a subsequence of $\widehat{\Gamma}_r(4)$, while ${\rm rev}(T_l)$ is the unique subsequence of ${\rm rev}({\Gamma}_l(4))$ consisting of the edge labels of a path in the quantum Bruhat graph (in fact, in the Bruhat graph) from the identity to $2\overline{4}\overline{3}\overline{1}5$. Note that the above path has a unique quantum edge, namely the last one. 

The filling map, defined in the same way as in type $C$ (see Definition \ref{deffillc}), sends $(w,T)$ to the filling
\[\sigma=C^rC^l=\tableau{{3}&{2}\\{\overline{2}}&{\overline{4}}\\{1}&{\overline{3}}\\{4}&{\overline{1}}}\,.\]
It turns out that this filling corresponds to the KN column 
\[D=\tableau{{3}\\{\overline{3}}}\;\;\mbox{in}\;\; B(\omega_2)\subset B^{4,1}\simeq B(\omega_4)\oplus B(\omega_2)\oplus B(\omega_0)\,.\]
To recover the filling $\sigma$, we first construct from $D$ an ``extended'' KN column 
\[\widehat{D}=\tableau{{1}\\{3}\\{\overline{3}}\\{\overline{1}}}\,,\]
following a procedure in \cite{schbtd}[Section 3.4]. We claim that the doubling procedure in Definition \ref{defkn} can be extended, but we still match the entries in $I=\{3>1\}$ with certain entries ``preceding'' them; the only difference is that this now requires us to go counterclockwise around the circle. So we obtain  $J=\{2,\overline{4}\}$, which gives the following splitting of $\widehat{D}$ by the usual rule:
\[\tau={\rm ord}(\sigma)=(r\widehat{D},l\widehat{D})=\tableau{{1}&{2}\\{3}&{\overline{4}}\\{4}&{\overline{3}}\\{\overline{2}}&{\overline{1}}}\,.\]
Finally, note that $\sigma$ is obtained from $\tau$ as usual, by setting $C_l=l\widehat{D}$, while $C^r$ is the reordering of $r\widehat{D}$ given by Condition \ref{cond2}$'$, see Remark \ref{remcondc}.

By examining the above path in the quantum Bruhat graph, we can see that ${\level}(w,T)=1$, which agrees with the energy of $D$ in $B^{4,1}$. This also agrees with the charge of $\tau$, computed via the sum in (\ref{chccomp}); indeed, there are two descents in $\sigma$, whose arm lengths are $1$, so ${\rm charge}(\tau)=\frac{1}{2}(1+1)=1$.
}
\end{example}

The second new aspect in type $B$ is the fact that not always a descent in the image of the filling map contributes half its arm length to the charge. We illustrate this with another example.

\begin{example}\label{exxx}{\rm Let $\mu=2\omega_2$ in type $B_3$, so $B_\mu=B^{2,1}\otimes B^{2,1}$. We have the $\mu$-chain $\Gamma=\widehat{\Gamma}_r(2)\Gamma_l(2)\widehat{\Gamma}_r(2)\Gamma_l(2)$. Consider $(w,T)=(w,T_l^2T_l^1)$ with $w=1{2}{3}$ (meaning that $T$ contains no roots in the two segments $\widehat{\Gamma}_r(2)$), which is uniquely determined by the path in the quantum Bruhat graph $\pi(w,T_l^2)$ below (read from right to left):
\[w=\begin{array}{l}\tableau{{1}\\{\mathbf {{2}}}}\\ \\ \tableau{{\mathbf {{3}}}} \end{array}\!<\!
\begin{array}{l}\tableau{{1}\\{\mathbf {\overline{3}}}}\\ \\ \tableau{{{\overline{2}}}} \end{array}\!>\!
\begin{array}{l}\tableau{{1}\\{\mathbf {{3}}}}\\ \\ \tableau{{\mathbf {\overline{2}}}} \end{array}\!<\!
\begin{array}{l}\tableau{{1}\\{{\overline{2}}}}\\ \\ \tableau{{ {{3}}}} \end{array}.\]

The filling map sends $(w,T)$ to the filling
\[\sigma=\tau=\tableau{{1}&{1}&{1}&{1}\\{2}&{2}&{\overline{2}}&{\overline{2}}}\,,\]
which is the doubling of the pair of KN columns
\[D^2D^1=\tableau{{1}&{1}\\{2}&{\overline{2}}}\]
in $B_\mu$. Note that ${\rm level}(w,T)=2$, as there are two quantum edges in $\pi(w,T)$, namely the first and the last one in $\pi(w,T_l^2)$ (as there are no quantum edges corresponding to $T_l^1$). This also agrees with the corresponding energy function. But in $\sigma$ we have only one descent of arm length 2, so we cannot use the sum in (\ref{chccomp}) to compute the charge. The problem is that $1\overline{2}3>123$ is not an edge of the type $B$ quantum Bruhat graph (although it is for the type $C$ one), so in order to change the entry $\overline{2}$ to $2$ we must perform more than a full rotation around the circle; namely we must use the intermediate entries $\overline{2}, 3,\overline{3},2$ in the second position of the signed permutation, and this results in two descents instead of one. This problem occurs because none of the entries between $\overline{2}$ and 2 in circular order, namely $\overline{1}$ and 1, are found after $\overline{2}$ in $1\overline{2}3$, in order to be accessible for the corresponding reflections.
}
\end{example}

The above example suggests that in type $B$ we need to modify the definition of charge by the sum in (\ref{chccomp}) as follows. We use the same notation $\sigma=C_r^{\mu_1}C_l^{\mu_1}\ldots C_r^1C_l^1$ as above. Let ${\rm Des}'(\sigma)$ denote the descents of the form $\overline{m}=C_r^j(i)>C_l^{j+1}(i)=m$ such that, for any $k=1,\ldots,m-1$, we have either $k$ or $\overline{k}$ in $C_r^j[1,i-1]$. We claim that the appropriate definition of the type $B$ charge is given by the following formula:
\begin{equation}\label{bcharge}\frac{1}{2}\sum_{c\in{\rm Des}(\sigma)\setminus{\rm Des}'(\sigma)}{\rm arm}(c)+\sum_{c\in{\rm Des}'(\sigma)}{\rm arm}(c)\,.\end{equation}
Moreover, there is an easy modification of the type $C$ algorithm for reconstructing the folding pair from its image under the filling map, which underlies the charge construction.

We said above that in type $D$ we have additional complexity still. We explain this using another example. Note that we now need to remove the roots/reflections $(i,\overline{\imath})$ from all constructions, which is precisely what leads to the mentioned complications. None of the situations presented below occur in type $B$.

\begin{example}{\rm Let $\mu=2\omega_1$ in type $D_3$. We consider a folding pair of the same form as the one Example \ref{exxx}, so we just indicate the corresponding path in the quantum Bruhat graph $\pi(w,T_l^2)$:
\[w=\begin{array}{l}\tableau{{\mathbf {\overline{3}}}}\\ \\ \tableau{{\mathbf {\overline{2}}}\\{1}}  \end{array}\!>\!
\begin{array}{l}\tableau{{\mathbf {{2}}}}\\ \\ \tableau{{{{3}}}\\{\mathbf {1}}}  \end{array}\!>\!
\begin{array}{l}\tableau{{\mathbf {{1}}}}\\ \\ \tableau{{\mathbf {{3}}}\\{{2}}}  \end{array}\!<\!
\begin{array}{l}\tableau{{{{3}}}}\\ \\ \tableau{{{{1}}}\\{{2}}}  \end{array}.\]
The corresponding filling $\sigma=\tableau{{\overline{3}}&{\overline{3}}&{3}&{3}}$ has no descents, but there is still a quantum edge in the above path, so the charge needs to be 1. The reason for this is that the values $3$ and $\overline{3}$ are incomparable in type $D_3$, so the pair $\tableau{{\overline{3}}&{3}}$ in $\sigma$ needs to contribute $1$ to the charge. In fact, the definition of charge needs to be adjusted even more, as the following more subtle example shows.

Let $\mu=2\omega_2$ in $D_4$. Once again, the considered folding pair is of the same form as the one Example \ref{exxx}, so we just indicate the corresponding path in the quantum Bruhat graph $\pi(w,T_l^2)$:
\[w=\begin{array}{l}\tableau{{\mathbf {\overline{4}}}\\{\mathbf {\overline{3}}}}\\ \\ \tableau{{ {\overline{2}}}\\{{\overline{1}}}} \end{array}\!>\!
\begin{array}{l}\tableau{{ {{3}}}\\{\mathbf {{4}}}}\\ \\ \tableau{{\mathbf {\overline{2}}}\\{{\overline{1}}}} \end{array}\!>\!
\begin{array}{l}\tableau{{ {{3}}}\\{\mathbf {{2}}}}\\ \\ \tableau{{ {\overline{4}}}\\{\mathbf {\overline{1}}}} \end{array}\!>\!
\begin{array}{l}\tableau{{ {{3}}}\\{\mathbf {{1}}}}\\ \\ \tableau{{\mathbf {\overline{4}}}\\{ {\overline{2}}}} \end{array}\!<\!
\begin{array}{l}\tableau{{ {{3}}}\\{ {\overline{4}}}}\\ \\ \tableau{{ {{1}}}\\{ {\overline{2}}}} \end{array}.\]
The corresponding filling, having only ascents or equal entries next to each other in a row, is 
\[\sigma=C_r^2C_l^2C_r^1C_l^1=\tableau{{\overline{4}}&{\overline{4}}&{3}&{3}\\{\overline{3}}&{\overline{3}}&{\overline{4}}&{\overline{4}}}\,.\]
However, there is a quantum edge in the above path, so the charge needs to be 1. The reason for which the charge is not 0, meaning that $\sigma$ is not a split KN tableau, is that the two middle columns in $\sigma$ represent a forbidden configuration in type $D_4$. (Recall from \cite{kancgr} that, in type $D$, in addition to the row monotonicity condition for the split tableau in types $B$ and $C$, there are extra conditions for a sequence of KN columns to form a KN tableau, and these are given in terms of certain forbidden configurations for a pair formed by a right column and the next left column.) 

The last example also highlights the need for an important modification in our algorithm for reconstructing the path from the filling $\sigma$. Indeed, note that the rightmost edge in the displayed path, which is part of a sequence of edges labeled $(2,l)$ or $(1,\overline{2})$, depends not only on $C_l^2(2)$, but also on $C_l^2(1)$ (recall that the path is reconstructed from the right). For instance, if we had 
\[\sigma=C_r^2C_l^2C_r^1C_l^1=\tableau{{{4}}&{{4}}&{3}&{3}\\{\overline{3}}&{\overline{3}}&{\overline{4}}&{\overline{4}}}\]
 instead, we would have chosen the edge labeled $(1,\overline{2})$, and this would have given the entire path $\pi(w,T_l^2)$ associated to $\sigma$. Such a situation did not occur in any of the previous examples. 
}
\end{example}


\end{document}